%% file: main.tex
\documentclass[hidelinks]{article}
\usepackage[left=3cm,right=3cm,top=2cm,bottom=2cm]{geometry} 
\usepackage{amsmath}
\usepackage{amsthm}
\usepackage{amssymb}
\usepackage{dsfont}
\usepackage{stackengine}
\usepackage{appendix}
\stackMath
\usepackage{hyperref}
\usepackage[svgnames]{xcolor}
\usepackage[style=numeric,backend=bibtex,natbib=true,maxnames=5]{biblatex}
\numberwithin{equation}{section}

\hypersetup{
	colorlinks = true,
	citecolor = teal, 
	linkcolor=red,
	urlcolor=blue}

\input{commands} 
\DeclareFieldFormat*{title}{#1}
\renewbibmacro{in:}{}
\DeclareFieldFormat*{eprint}{\href{http://arxiv.org/abs/#1}{#1}}
\begin{document}
\begin{minipage}{0.85\textwidth}
	\vspace{2.5cm}
\end{minipage}
\begin{center}
	\large\bf Mesoscopic eigenvalue statistics for Wigner-type matrices
\end{center}

\vspace{0.5cm}

\renewcommand{\thefootnote}{\fnsymbol{footnote}}	
\begin{center}
	Volodymyr Riabov\footnotemark[1]\\
	\footnotesize 
	{Institute of Science and Technology Austria}\\
	{\it volodymyr.riabov@ist.ac.at}
\end{center}

\bigskip

\footnotetext[1]{Supported by the ERC Advanced Grant "RMTBeyond" No.~101020331}
\renewcommand*{\thefootnote}{\arabic{footnote}}
\vspace{0.5cm}

\begin{center}
	\begin{minipage}{0.91\textwidth}\footnotesize{
			{\bf Abstract.}}
		We prove a universal mesoscopic central limit theorem for linear eigenvalue statistics of a Wigner-type matrix inside the bulk of the spectrum with compactly supported twice continuously differentiable test functions. The main novel ingredient is an optimal local law for the two-point function $T(z, \zeta)$  and a general class of related quantities involving two resolvents	at nearby spectral parameters. 
	\end{minipage}
\end{center}

\vspace{0.8cm}

{\small
	\footnotesize{\noindent\textit{Date}: \today}
}

\vspace{10mm}

\thispagestyle{headings}
\section{Introduction}
	In the study of the eigenvalue distribution of large random matrices, the most celebrated analog of the Law of Large Numbers is the Wigner semicircle law \cite{Wigner1957}. It states that the empirical density of eigenvalues converges to a deterministic limit known as the semicircle distribution $\rho_{sc}$. More explicitly, if $H$ is an $N\times N$ Wigner matrix and $f$ is a sufficiently smooth test function, then the \textit{linear eigenvalue statistics} $N^{-1}\Tr f(H)$ converge in probability to $\int_\mathbb{R}f(x)\rho_{sc}(x)\mathrm{d}x$ in the large $N$ limit.
		
	The  corresponding Central Limit Theorem (CLT) asserts that the asymptotic fluctuations of the linear eigenvalue statistics $\Tr f(H) - \E{\Tr f(H)}$ are Gaussian. The absence of the $N^{-1/2}$ normalization factor, appearing in the classical CLT, can be viewed as a manifestation of the strongly-correlated nature of the eigenvalues.
	For the special case of $f(x) = (x-z)^{-1}$ with $\im z \neq 0$, this result was obtained by Khorunzhy, Khoruzhenko and Pastur \cite{Khorunzhy1996}. Johansson obtained the CLT for invariant ensembles with arbitrary polynomial potentials in \cite{Johansson1998}.  In \cite{Bai2005}, Bai and Yao used martingale CLT to establish the result for Wigner matrices with analytic test functions.  The proof for bounded test functions $f$ with bounded derivatives appeared in the work of Lytova and Pastur \cite{Lytova2009indCLT}. In subsequent works, different moment conditions on the matrix and regularity conditions on the test function were studied extensively by many authors, e.g., \cite{Bao2016, Landon2022, Shcherbina2011, Sosoe2013}. 
	
	While fixed test functions represent macroscopic averaging in the spectrum, one can introduce $N$-dependent scaling and consider \textit{ scaled  test functions} of the form $f(x) = g(\eta_0^{-1}(x-E_0))$, where $E_0$ is a fixed reference energy in the bulk, $\eta_0 \equiv \eta_0(N) \ll 1$  is a scaling parameter, and $g$ is compactly supported. Then $\Tr f(H)$ involves only about $N\eta_0$ eigenvalues of $H$. In particular, on \textit{mesoscopic} scales, corresponding to $N^{-1}\ll\eta_0\ll 1$, the limiting variance is given by the square of the $\dot H^{1/2}$ norm of $g$.  Mesoscopic test functions were first studied by Boutet de Monvel and Khorunzhy in \cite{Khorunzhy1999} for the Gaussian Orthogonal Ensemble, with subsequent extension to real Wigner matrices in \cite{Boutet1999} with $N^{-1/8}\ll\eta_0\ll 1$.  
	In \cite{He2017WignerCLT}, He and Knowles proved the CLT for Wigner matrices with general mesoscopic test functions for all scaling parameters $N^{-1}\ll\eta_0\ll 1$.  
	
	The result was extended to ensembles of greater generality in the more recent works, see, e.g., \cite{Bao2020freesum} and \cite{Li2021deformed}.
	In particular, Li and Xu obtained mesoscopic CLT for \textit{generalized Wigner matrices} \footnote[1]{
		Generalized Wigner matrices are characterized by a flat doubly-stochastic matrix of variances $S$. Unlike the Wigner case, the entries $S_{jk}$  are not assumed to be equal. The limiting eigenvalue distribution remains semicircular.
	} 
	in the bulk and at the spectral edge with $C^2_c$ test functions in the full range of scales \cite{Li2021genWigner}.

	Finally, Landon, Lopatto, and Sosoe proved the bulk CLT for the much more general ensemble of \textit{Wigner-type matrices} in \cite{Landon2021Wignertype} for two classes of $C^\infty$ test functions. 
	\new For a special class of globally supported regularized bump functions \old, the proof is performed via resolvent techniques for large scales and extended to the entire mesoscopic range using Dyson Brownian motion (DBM) dynamics. \new For the more conventional compactly supported \old scaled 
	test functions, the bulk CLT is established on all mesoscopic scales $N^{-1} \ll \eta_0 \ll 1$ using a combination of DBM and Green's function comparison.  
 	
	Wigner-type matrices were first introduced in \cite{Ajanki2015Wignertype}; they have centered entries $H_{jk}$ independent up to the symmetry constraint $H=H^*$. The matrix of variances $S$, defined by $S_{jk} := \E{|H_{jk}|^2}$, is assumed to be \textit{flat}, i.e., $S_{jk}\sim N^{-1}$ and satisfy a piece-wise H\"older regularity condition (see \eqref{cond_B}). 
	
	 As the main step towards CLT in the present paper,  we prove the optimal averaged and entry-wise local laws  (Corollary \ref{cor_T_local_law})  for the two-point function $T$, defined by
	\begin{equation} \label{Tfunction}
		T_{xy}(z,\zeta) := \sum_{a\neq y} S_{xa} G_{ay}(z)G_{ya}(\zeta),\quad x,y \in \{1,\dots,N\},
	\end{equation}
	where $G(z)$ is the resolvent of $H$.
	The corresponding result in the simpler setting of generalized Wigner matrices was obtained in \cite{Li2021genWigner}. Using the optimal local law for $T(z,\zeta)$, we prove the bulk mesoscopic CLT for Wigner-type matrices in the full range of scales $N^{-1} \ll \eta_0 \ll 1$ 	for compactly supported $C^2$ scaled test functions (Theorem \ref{main_Th}).  Our proof relies entirely on resolvent methods, circumventing the DBM dynamics used in \cite{Landon2021Wignertype}.

	Understanding $T(z,\zeta)$ is the crucial ingredient for the CLT as it was realized in \cite{Landon2021Wignertype}. In fact, a suboptimal entry-wise local law for  $T_{xy}(z,\zeta)$  was proved in Proposition 5.1 of \cite{Landon2021Wignertype}. If one relies solely on resolvent methods, this local law provides sufficient control for mesoscopic CLT only on scales $\eta_0 \gg N^{-1/5}$. The main reason for this limitation is that the error term in \cite{Landon2021Wignertype} contains the norm of the inverted \textit{stability operator} (defined in \eqref{Smm_def}). In the present paper, we show that this factor can be removed by separating the \textit{destabilizing eigendirection} corresponding to the smallest eigenvalue of the stability operator.  Using this method, we prove a local law for a general class of  quantities involving two resolvents (Theorem \ref{th_T_local_law}) and deduce the optimal averaged and entry-wise local laws for $T(z,\zeta)$.  In particular, this allows us to obtain the CLT on all mesoscopic scales  without relying on DBM. 
	
	The main difficulty lies in the fact that the deterministic approximation of the resolvent for Wigner-type matrices is not a multiple of the identity matrix, contrary to the generalized Wigner case \cite{Li2021genWigner}. Consequently, the destabilizing direction is no longer parallel to the vector of ones, and generally, no closed-form expression is known for the corresponding eigenprojector.
	It is important to note that for the deformed Wigner matrices studied in \cite{Li2021deformed}, the deterministic approximation is also not a multiple of the identity, but \new $S_{jk} = N^{-1}$\old . Therefore, the two-point function can be expressed as the square of the resolvent and can be studied using the local law, similarly to the standard Wigner case.
	
	Instead of approximating the destabilizing direction to circumvent this difficulty, we use a contour integral representation for the eigenprojector. It allows us to extend the decomposition approach of \cite{Li2021genWigner} to the Wigner-type ensembles. This method benefits from yielding an integral representation for the variance on all mesoscopic scales, under weaker regularity conditions on the test function than in \cite{Landon2021Wignertype}, and relying only on resolvent methods. 
	
		
	 The paper is organized in the following way. Section \ref{result_sect} contains the precise definition of the model and the statement of our main \new mesoscopic CLT \old result, Theorem \ref{main_Th}. In Section \ref{main_tech_sect}, we present our main technical result, \new the optimal local law for two-point functions in \old Theorem \ref{th_T_local_law}. In Section \ref{prelim_sect}, we collect notations and preliminary results to which we refer throughout the paper. In Section \ref{proof_sect}, we deduce Theorem \ref{main_Th} from Propositions \ref{main1} and \ref{main2}, and prove Proposition \ref{main1} using a local law for $T(z,\zeta)$ (Corollary \ref{cor_T_local_law}) as an input. The proofs of Theorem \ref{th_T_local_law} and Corollary \ref{cor_T_local_law} are presented in Section \ref{Curly T section}. In Section \ref{proof_of_main2}, we prove Proposition \ref{main2}, which relates the variance of the linear eigenvalue statistics to the $\dot H^{1/2}$-norm.

	\textbf{Acknowledgments.} \\
	I would like to express my gratitude to L\'aszl\'o Erd\H{o}s for suggesting the project and supervising my work. I am also thankful to Yuanyuan Xu
	and Oleksii Kolupaiev for many helpful discussions.
	\newpage
\section{Model and Main Result} \label{result_sect}
	We begin with the definition of Wigner-type matrices originally introduced in Section 1.1 of \cite{Ajanki2015Wignertype}.
	\begin{Def}[Wigner-type matrices] \label{WT_def}
		Let $H = \left(H_{jk} \right)_{j,k=1} ^N $ be an $N\times N$ matrix with independent entries up to the Hermitian symmetry condition $H = H^*$ satisfying 
		\begin{equation}
			\E{H_{jk}} = 0.
		\end{equation}
		We consider both real and complex Wigner-type matrices. In case the matrix $H$ is complex we assume additionally that $\re H_{jk}$ and $\im H_{jk}$ are independent and $\Expv[H_{jk}^2] = 0$ for $k\neq j$.
		
		Denote by $S$ the matrix of variances $S_{jk} := \Expv[|H_{jk}|^2]$, and assume it satisfies
		\begin{gather} \tag{A} \label{cond_A}
			\frac{c_{inf}}{N} \le S_{jk} \le \frac{C_{sup}}{N},
		\end{gather}
		for all $j,k \in \{1,\dots, N\}$ and some strictly positive constants $C_{sup}, c_{inf}$.
		
		We assume a uniform bound on all other moments of $\sqrt{N}H_{jk}$, that is, for any $p \in \mathbb{N}$ there exists a positive constant $C_p$ such that
		\begin{equation} \label{moment_condition}
			\E{|\sqrt{N}H_{jk}|^p} \le C_p
		\end{equation}
		holds for all $j,k \in \{1,\dots, N\}$.
		
		Additionally, we assume that $S$ satisfies a
		H\"older regularity condition\footnote{As stated in \cite{Ajanki2015Wignertype}, assumption \eqref{cond_B} can be weakened to
			piece-wise $1/2$-H\"older regularity condition for some positive constant $L$ on finitely many intervals, in the sense that
			\begin{gather*}
				\max\limits_{a,b}\max\limits_{j,j'\in(NI_b)} \max\limits_{k,k' \in (NI_a)}  N^{3/2}\frac{|S_{jk} - S_{j'k'}|}{|j-j'|^{1/2}+|k-k'|^{1/2}}\le L,
			\end{gather*}
			where $\{I_a\}_{a=1}^n$ is a fixed finite partition of $[0,1]$ into smaller intervals, and $(NI_a)$ denotes the set of positive integers $j$ such that $j/N$ lies in $I_a$. 
		}, that is,
		\begin{gather} \tag{B} \label{cond_B}
			|S_{jk} - S_{j'k'}| \le \frac{L}{N}\left( \frac{|j-j'|+|k-k'|}{N}\right)^{1/2},
		\end{gather}
		for all $j,j',k,k' \in \{1,\dots,N\}$ and some positive constant $L$. The constants $c_{inf}$, $C_{sup}$, $C_p$ and $L$ are independent of $N$. 
	\end{Def}
	\subsection{Central Limit Theorem for Mesoscopic Linear Eigenvalue Statistics}
	\begin{theorem} \label{main_Th} (c.f. Theorem 2.5 in \cite{Landon2021Wignertype})
		Let $g$ be a $C^2_c(\mathbb{R})$ test function. Let $\varepsilon_0$ be a small fixed constant and let $N^{-1+\varepsilon_0} \le \eta_0 \le N^{-\varepsilon_0}$, and let $E_0$ be a fixed reference energy in the bulk of the spectrum, that is,  $\rho(E_0) \ge \varepsilon_0$ (\new here $\rho$ is the density of states to be defined in \eqref{rho_def} below \old).
		Define the scaled test function $f$ to be 
		\begin{equation} \label{scaled_f}
			f(x) := g\left(\frac{x-E_0}{\eta_0}\right),
		\end{equation} 
		then
		\begin{equation} \label{lim_gauss}
			\Tr f(H) - \E{\Tr f(H)} \xrightarrow{d} \mathcal{N}\left(0,\frac{1}{2\beta\pi^2}\norm{g}_{\dot{H}^{1/2}}^2\right),
		\end{equation}
		where $\beta =1$ and $\beta = 2$ corresponds to real symmetric and complex Hermitian $H$, respectively.
	\end{theorem}
	
	\begin{remark}
		We remark that the universal limiting variance in \eqref{lim_gauss} coincides with the corresponding formulas for standard Wigner matrices \cite{He2017WignerCLT}, where $S_{jk} = N^{-1}$, $m_j(z) = m_{sc}(z)$ for all $j,k \in \{1,\dots, N\}$, and $m_{sc}(z)$ is the Stieltjes transform of the semicircle law.
	\end{remark}
	
\section{Local Laws for the Two-point Functions} \label{main_tech_sect}
	In this section, we introduce our main technical result, local laws for quantities that involve two resolvents of a Wigner-type matrix. Our prime motivation is to study the function $T(z,\zeta)$ defined in \eqref{Tfunction}, but our methods allow us to estimate a more general class of quantities, namely
	\begin{equation} \label{twoGdef}
		\sum_{a \neq y} w_a G_{\alpha a}(z) G_{a\beta}(\zeta), \quad \sum_b\sum_{a \neq b} W_{ab}G_{ba}(z)G_{ab}(\zeta),
	\end{equation}
	\new for fixed indices $\alpha,\beta,y$, and \old deterministic weights $w_a$, $W_{ab}$ satisfying $|w_a|, |W_{ab}|\le cN^{-1}$	 for some constant $c>0$. Here $G(z) := \left(H-z\right)^{-1}$ denotes the resolvent of $H$.  Objects of this type were first studied in \cite{Erdos2013DiffProfile} in the setting of random band matrices. We obtain the estimates in the sense of stochastic domination.
	\begin{Def} (Definition 2.1 in \cite{Erdos2013LocalLaw})  \label{stochdom_def}
		Let $\mathcal{X} = \mathcal{X}^{(N)}(u)$ and $\mathcal{Y} = \mathcal{Y}^{(N)}(u)$ be two families of random variables possibly depending on a parameter $u \in U^{(N)}$. We say that $\mathcal{Y}$ stochastically dominates $\mathcal{X}$ uniformly in $u$ if for any $\varepsilon > 0$ and $D>0$ there exists $N_0(\varepsilon,D)$ such that for any $N \ge N_0(\varepsilon, D)$,
		\begin{equation*}
			\sup\limits_{u\in U^{(N)}}\Prob{\mathcal{X}^{(N)}(u) > N^\varepsilon \mathcal{Y}^{(N)}(u)} < N^{-D}.
		\end{equation*}
		We denote this relation by $\mathcal{X} \prec \mathcal{Y}$ or $\mathcal{X} = \Oprec(\mathcal{Y})$.
	\end{Def}
	We consider spectral parameters $z$ lying in the domain $\mathcal{D}$, defined by
	\begin{equation} \label{D_def}
		\mathcal{D} := \{z \in\mathbb{C} : N^{-1+\tau} \le |\im z| \le \tau^{-1}, \, |\re z| \le \tau^{-1} \},
	\end{equation}
	for a fixed $\tau > 0$. As in Theorem \ref{main_Th}, our analysis is limited to the bulk of the spectrum, which we define via the \textit{self-consistent density of states} $\rho(E)\equiv\rho_N(E)$. The density $\rho(E)$ is recovered by the Stieltjes inversion formula, 
	\begin{equation} \label{rho_def}
		\rho(E) := \pi^{-1}\lim_{\eta\to+0}\im m(E+\I\eta),
	\end{equation}
	where $m(z) := N^{-1}\sum_{j=1}^N m_j(z)$, and $\m(z) = (m_j(z))_{j=1}^N$ is the unique (Theorem 4.1 in \cite{Ajanki2015Wignertype}) solution to the \textit{vector Dyson equation}
	\begin{equation} \label{VD_eq}
		\frac{-1}{\m(z)} = z + S\m(z),\quad \im \m(z)\im{z} > 0.
	\end{equation}
	Let $\bulk$ be the set on which $\rho(E)$ is positive. Theorem 4.1 of \cite{Ajanki2015Wignertype} guarantees that $\bulk$ consists of a finite union of open intervals $(\al^{(j)}, \be^{(j)})$. Then for $\kappa > 0$, we define the \textit{bulk domain} by
	\begin{equation} \label{D_bulk}
		\mathcal{D}_\kappa := \{z\in \mathcal{D} : \re z \in \bulk_\kappa\}, \quad 	\bulk_\kappa := \bigcup\limits_j [\al^{(j)}+\kappa, \be^{(j)} - \kappa].
	\end{equation}
	In particular, for all $z\in\mathcal{D}_\kappa$, $\rho(z) \ge C(\kappa)$
	for some constant $C(\kappa)>0$. \new Given $E_0$ as in Theorem \ref{main_Th}, we \old choose $\kappa$ so that $E_0 \in \bulk_{2\kappa}$.
	\begin{theorem} \label{th_T_local_law}
		There exists a positive constant $\epsilon=\epsilon_\kappa$ which is independent of $N$, such that for all $z,\zeta$ in $\mathcal{D}_\kappa$ with $|\re\zeta - \re z|\le \epsilon$, and deterministic vectors $\vect{w}\in\mathbb{C}^N$ satisfying $\norm{\vect{w}}_\infty \le cN^{-1}$, the following estimate holds, 
		\begin{equation} \label{twoG_law_1}
			\begin{split}
				\sum_{a\neq y} w_{a} G_{\alpha a}(z) G_{a\beta}(\zeta) 
				=&\,\delta_{\alpha \beta}\bigl[\m(z)\m(\zeta) \bigl(1- S\m(z)\m(\zeta) \bigr)^{-1}\vect{w} \bigr]_{\alpha} - \delta_{\alpha \beta}\delta_{\alpha y}[\m(z)\m(\zeta)\vect{w}]_{\alpha}\\
				&+ \Oprec \bigl((\Psi(z)+\Psi(\zeta)) (\Psi(z)\Psi(\zeta) + \mathds{1}_{\{\im z\im \zeta <0\}}\min\{\Theta(z),\Theta(\zeta)\})\bigr), 
			\end{split} 
		\end{equation}
		where the vector $\m$ is identified with the diagonal operator $\diag{\m}$.
		
		Under the same conditions on $z$, $\zeta$, for any deterministic $N \times N$ matrix $W$ satisfying $|W_{ab}| \le cN^{-1}$ for all $a,b$, the following estimate holds,
		\begin{equation} \label{twoG_law_2}
			\begin{split}
				\sum_b\sum_{a\neq b} W_{ab}G_{ba}(z)G_{ab}(\zeta) =& \Tr\bigl[\m(z)\m(\zeta)S\m(z)\m(\zeta)\bigl(1- S\m(z)\m(\zeta) \bigr)^{-1}W \bigr]\\
				&+ N\Oprec \bigl((\Psi(z)+\Psi(\zeta))\Psi(z)\Psi(\zeta)+\mathds{1}_{\{\im z\im \zeta <0\}}\Theta(z)\Theta(\zeta)\bigr).
			\end{split}
		\end{equation}
		Here $\Psi(z)$ and $\Theta(z)$ denote control parameters defined as 
		\begin{equation} \label{PsiTheta_def}
			\Psi(z) := \sqrt{\frac{|\im m(z)|}{N|\eta|}} + \frac{1}{N|\eta|}, \quad \Theta(z) := \frac{1}{N|\eta|}, \quad z = E + \I\eta \in \mathbb{C}\backslash\mathbb{R}.
		\end{equation}
	\end{theorem}
	Theorem \ref{th_T_local_law} implies the following averaged and entry-wise local laws for $T(z,\zeta)$ \new from \eqref{Tfunction} \old.
	\begin{corollary} \label{cor_T_local_law}
		Let $z,\zeta$ satisfy the assumptions of Theorem \ref{th_T_local_law}. The entries $T_{xy}(z,\zeta)$ admit the estimate
		\begin{equation} \label{T_xy_law}
			\begin{split}
					T_{xy}(z,\zeta) =& \bigl[\left(S\m(z)\m(\zeta)\right)^2\bigl(1- S\m(z)\m(\zeta) \bigr)^{-1} \bigr]_{xy}\\
					&+ \Oprec \bigl((\Psi(z)+\Psi(\zeta))\bigl(\Psi(z)\Psi(\zeta)+\mathds{1}_{\{\im z\im \zeta <0\}}\min\{\Theta(z),\Theta(\zeta)\}\bigr)\bigr).
				\end{split}
		\end{equation}
		Furthermore, for all deterministic $N \times N$ matrices $A$, the following equality holds
		\begin{equation} \label{T_local_law}
			\begin{split}
				\Tr[A\,T(z,\zeta)] =&\Tr[A\bigl(1-S\m(z)\m(\zeta)\bigr)^{-1}\bigl(S\m(z)\m(\zeta)\bigr)^{2}]
				\\&+ N\opnorminf{A}\Oprec\bigl((\Psi(z)+\Psi(\zeta))\Psi(z)\Psi(\zeta) + \mathds{1}_{\{\im z\im \zeta <0\}}\Theta(z)\Theta(\zeta)\bigr).
			\end{split} 
		\end{equation}
	\end{corollary}
	\begin{remark}
		The error estimates in the entry-wise local law $\eqref{twoG_law_1}$, and hence in \eqref{T_xy_law} are optimal. Indeed, for $S_{jk} := N^{-1}$, which corresponds to the standard Wigner matrices, and $\zeta = \bar z$, a simple calculation using the Ward identity shows that
		\begin{equation}
			T_{xy}(z,\bar z) = N^{-1}|\im z|^{-1}\im m_{sc}(z) - N^{-1}|m_{sc}(z)|^2 + \Oprec\bigl(\Theta(z)\Psi(z)\bigr).
		\end{equation}
		The error estimate in \eqref{twoG_law_2} is not optimal; it can be improved to
		\begin{equation} \label{optimal_avgT_error}
			\Oprec \bigl(N(\Psi(z)+\Psi(\zeta))^2\bigl(\Psi(z)\Psi(\zeta)+\mathds{1}_{\{\im z\im \zeta <0\}}N\Theta(z)\Theta(\zeta)\bigr)\bigr)
		\end{equation}
		However, \eqref{twoG_law_2} 
		is sufficient for establishing the CLT, so for the sake of brevity, we do not present the proof of \eqref{optimal_avgT_error} in full detail. We only indicate the necessary ingredients in Remark \ref{optimal_T_rem} below. 	
	\end{remark}

\section{Notations and Preliminaries} \label{prelim_sect}
	\subsection{Notations}
	For a vector $\vect{x} = (x_j)_{j=1}^N \in \mathbb{C}^N$ we use the standard definitions of $\ell^2$ and $\ell^\infty$ norms, namely,
	\begin{equation*}
		\norm{\vect{x}}_2 =  \biggl(\sum\limits_{j=1}^N |x_j|^2 \biggr)^{1/2}  ,\quad \norm{\vect{x}}_\infty = \max\limits_{j}|x_j|.
	\end{equation*}
	For a linear operator $T :\mathbb{C}^N \to \mathbb{C}^N$, we denote its matrix norms induced by $\ell^2$ and $\ell^\infty$ norms, respectively, by 
	\begin{equation*}
		\opnormtwo{T} = \sup\limits_{\norm{\vect{x}}_2 = 1}\norm{T\vect{x}}_2,\quad \opnorminf{T} = \sup\limits_{\norm{\vect{x}}_\infty = 1}\norm{T\vect{x}}_\infty.
	\end{equation*}
	For two vectors $\vect{x},\vect{y}\in \mathbb{C}^N$ we use angle brackets to denote the $\ell^2$ scalar product, while for a single vector $\vect{x} \in \mathbb{C}^N$ angle brackets denote the average of its coordinates $$\langle \vect{x},\vect{y} \rangle = \sum\limits_{j=1}^N \bar x_j y_j,\quad \langle \vect{x} \rangle = \frac{1}{N}\sum\limits_{j=1}^N x_j.$$
	We use $\vect{x}\vect{y}$ to denote a coordinate-wise product of vectors $\vect{x}$ and $\vect{y}$,
	\begin{equation*}
		(\vect{x}\vect{y})_j = x_jy_j, \quad j \in \{1,\dots,N\}.
	\end{equation*}
	Similarly, for a given vector $\vect{x}$ with non-zero entries, $\ddfrac{1}{\vect{x}}$ denotes a coordinate-wise multiplicative inverse
	\begin{equation*}
		\left(\frac{1}{\vect{x}}\right)_j = \frac{1}{x_j}, \quad j\in \{1,\dots,N \}.
	\end{equation*}
	We use $\vect{1}$ to denote the vector of ones  $(1,\dots,1)^t$ in $\mathbb{C}^N$.
	
	For a measurable function $f : \mathbb{R} \to \mathbb{R}$ we use the standard definition of the $L^p$ norms for $p \ge 1$, and the following definition of the $\dot{H}^{1/2}$ norm
	\begin{equation*}
		\norm{f}_{\dot{H}^{1/2}} = \left(\iint\limits_{\mathbb{R}^2} \frac{|f(x)-f(y)|^2}{|x-y|^2}\mathrm{d}x\mathrm{d}y\right)^{1/2}.
	\end{equation*}
	For two deterministic quantities $X, Y \in \mathbb{R}$ depending on $N$, we write $X \ll Y$ if there exists $\varepsilon, N_0 > 0$ such that $|X| \le N^{-\varepsilon}|Y|$ for all $N \ge N_0$.
	Similarly, we write $X \lesssim Y$ if there exists a constant $C,N_0>0$ such that $|X| \le C|Y|$ for all $N \ge N_0$, and $X \sim Y$ if both $X\lesssim Y$ and $Y \lesssim X$ hold.

	We use $C$ and $c$ to denote constants, the precise value of which is irrelevant and may change from line to line.

	\subsection{Local Law for the Resolvent}
	In this subsection, we summarize the facts on Wigner-type matrices that we use throughout our proofs. Majority of these results were obtained in \cite{Ajanki2019QVE} (see also \cite{Ajanki2017SingularitiesQVE}), but we refer to their concise versions from \cite{Ajanki2015Wignertype} adapted for the Wigner-type setting.
    
    \begin{lemma} \label{m_lemma}(
    	Theorem 4.1 in \cite{Ajanki2015Wignertype}) The solution $\m(z)$ of \eqref{VD_eq} satisfies the following properties:\\
        (1) For every $j \in \{1,\dots, N \}$ there exists a generating probability measure $\nu_j(\mathrm{d}x)$ such that
        \begin{equation} \label{nu_def}
            m_j(z) = \int\limits_{\mathbb{R}} \frac{\nu_j(\mathrm{d}x)}{x-z}.
        \end{equation}
        (2) If the matrix of variances $S$ satisfies conditions \eqref{cond_A} and \eqref{cond_B}, then for all $z \in \mathbb{C}\backslash\mathbb{R}$, the solution admits the following bounds
        \begin{equation} \label{m_bound}
            \norm{\m(z)}_\infty \le 
            \frac{c}{1+|z|}, \quad \norm{\frac{1}{\m(z)}}_\infty \le C(1+|z|).
        \end{equation}
    \end{lemma} 

	We now state the optimal \textit{averaged and isotropic local laws} for Wigner-type matrices.
    \begin{theorem} \label{local_law} (Corollary 1.8 in \cite{Ajanki2015Wignertype})
        Let $\vect{w},\vect{x},\vect{y}$ be deterministic vectors in $\mathbb{C}^N$ satisfying $\norm{\vect{w}}_\infty = 1$ and $\norm{\vect{x}}_2 = \norm{\vect{y}}_2 = 1$. Then the following estimates hold uniformly in $z\in\mathcal{D}$:
        \begin{equation} \label{law2}
        	N^{-1}\bigl| \Tr\bigl[\vect{w} (G(z) - \m(z))\bigr] \bigr| \prec \Theta(z), \quad  \bigl|\langle \vect{x}, (G(z) -\m(z)) \vect{y} \rangle \bigr| \prec \Psi(z),
        \end{equation} 
    	where vectors $\m$ and $\vect{w}$ are associated with corresponding diagonal matrices.
    \end{theorem}
	 In particular, it follows from  the isotropic local law \eqref{law2}  that for any $j$, $k \in \{1,\dots, N\}$,
	\begin{equation} \label{law1}
	\left|G_{jk}(z) - \delta_{jk} m_j(z) \right| \prec \Psi(z).
	\end{equation}

	\subsection{Preliminary Bounds on the Stability Operator}
	A significant part of our proof revolves around the \textit{stability operator}, originally introduced in \cite{Ajanki2019QVE}, that emerges when studying the two-point function $T(z,\zeta)$ defined in \eqref{Tfunction}. In this subsection, we collect the known bounds on the stability and related operators.
	
    The stability operator $\left(1 -S\m(z)\m(\zeta)\right)$ is defined by the matrix with entries
    \begin{equation} \label{Smm_def}
        \left(1-S\m(z)\m(\zeta)\right)_{jk} :=\delta_{jk} - S_{jk}m_k(z)m_k(\zeta),\quad j,k\in \{1,\dots, N\},\quad z,\zeta \in \mathbb{C}\backslash\mathbb{R}.
    \end{equation} 
	Throughout this paper we use $\m$ (and various functions of $\m$, such as $\im\m$, $|\m|$, $\m^{-1}$, $\m'$) to denote both a vector $(m_j)_{j=1}^N$ 
	and the corresponding multiplication operator, i.e., $\diag{(m_j)_{j=1}^N}$. Note that this notation agrees with the point-wise multiplication of two vectors if the first multiplicand is interpreted as an operator. We stress which interpretation is used whenever ambiguity may arise.
	
    The analysis of the stability operator relies on the corresponding \textit{saturated self-energy operator} $F$,  studied in \cite{Landon2021Wignertype}, that depends on two spectral parameters $z$, $\zeta$, and is defined as 
    \begin{equation} \label{F_def}
        F_{jk}(z,\zeta) := |m_j(z)m_j(\zeta)|^{1/2}S_{jk}|m_k(z)m_k(\zeta)|^{1/2},\quad j,k\in \{1,\dots, N\},\quad z,\zeta \in \mathbb{C}\backslash\mathbb{R}.
    \end{equation}
    The following statements encompass the main properties of $F$ and preliminary bounds on the stability operator.
    \begin{prop} \label{F_prop} (Proposition 4.3 in \cite{Landon2021Wignertype}, c.f. Proposition 7.2.9 and Lemma 7.4.4 in \cite{Erdoes2019dyson})
    For any $z,\zeta \in \mathbb{C}$, the principal eigenvalue of $F$ defined in \eqref{F_def} is positive and simple, the corresponding $\ell^2$-normalized eigenvector $\vv(z,\zeta)$ has strictly positive entries.
    The norm of $F$ admits the following upper bound
        \begin{equation}
            \opnormtwo{F(z,\zeta)} \le 1 - \frac{1}{2}\biggl( |\im z| \frac{\langle \vv(z,z), |\m(z)|\rangle}{\langle \vv(z,z), \frac{|\im \m(z)|}{|\m(z)|}\rangle} +
            |\im \zeta| \frac{\langle \vv(\zeta,\zeta), |\m(\zeta)| \rangle}{\langle \vv(\zeta,\zeta), \frac{|\im \m(\zeta)|}{|\m(\zeta)|} \rangle}
            \biggr).
        \end{equation}
     If $|z|,|\zeta| \lesssim 1$, then the entries of $\vv(z,\zeta)$ are comparable in size, that is
    \begin{equation} \label{v_bound}
    	c_\kappa\le \sqrt{N} v_j(z,\zeta) \le C_{\kappa}, \quad j \in \{1,\dots, N\}, 
    \end{equation}
    and moreover, let $\Gap{F}$ denote the difference between the two largest eigenvalues of $|F| = \sqrt{FF^*}$, then $\Gap{F}$ admits the bound
    \begin{equation} \label{gapF}  
    	\Gap{F} \ge \other{\delta},
    \end{equation}
    where $\other \delta$ is a constant that depends only on the constants in conditions \eqref{cond_A}, \eqref{cond_B} and $\kappa$.
    
    Furthermore, for a fixed $\kappa>0$ and $z,\zeta \in \mathcal{D}_\kappa$ there exists a positive constant $\other{c}_\kappa$ such that
        \begin{equation} \label{F_norm}
            \opnormtwo{F(z,\zeta)} \le 1 - \other c_\kappa\left(|\im z| + |\im \zeta| \right),
        \end{equation}
    \end{prop}
    \begin{prop} (Proposition 4.6 and Lemma 4.7 in \cite{Landon2021Wignertype}) \label{Smm_prop} Let $z,\zeta \in \mathbb{C}$, such that $|z|,|\zeta| \lesssim 1$ and $\re z, \re\zeta \in \bulk_\kappa$, then
	\begin{equation} \label{Smm_bound}
		\opnormtwo{(1-S\m(z)\m(\zeta))^{-1}} + \opnorminf{(1-S\m(z)\m(\zeta))^{-1}}\lesssim \frac{1}{|\im z| + |\im \zeta|}.
	\end{equation}
	 If additionally $\im z \im\zeta > 0$, the estimate is improved to
	\begin{equation} \label{norm in same half plane}
		\opnorminf{(1-S\m\other\m)^{-1}} \le C_\kappa,
	\end{equation}
	where $C_\kappa>0$ is a positive constants dependent on $\kappa$. 
	\end{prop}
	Finally, we state the bounds on the stability operator in the special case of $\zeta = z$, which is related to the derivative of $\m$ via the (vector) identity $\m'(z) = (1-\m^2(z)S)^{-1}\m^2(z)$, obtained by taking the derivative of \eqref{VD_eq}.
	\begin{lemma} (Lemma 5.9 in \cite{Ajanki2019QVE}, Lemma 7.3.2 in \cite{Erdoes2019dyson}) Let $C>0$ be a positive constant, then for $z\in\mathbb{C}\backslash\mathbb{R}$ with $|z| \le C$ we have 
	\begin{equation} \label{m2S_bound}
		\opnormtwo{(1-\m^2(z)S)^{-1}} + \opnorminf{(1-\m^2(z)S)^{-1}} \lesssim |\rho(z)|^{-2},
	\end{equation}
	where $\rho(z) = \pi^{-1}\langle\im\m(z)\rangle$ is the harmonic extension of $\rho(E)$ defined in \eqref{rho_def}.
	\end{lemma}
	Therefore for all $z\in\mathbb{C}\backslash\mathbb{R}$ with $\re z\in \bulk_\kappa$ we have
	\begin{equation} \label{dm_bound}
		\norm{\m'(z)}_\infty \lesssim 1.
	\end{equation}

	\subsection{Cumulant Expansion Formula}
    \begin{lemma}(Section II in \cite{Khorunzhy1999}, Lemma 3.1 in \cite{He2017WignerCLT}) \label{cumulant_lemma}
        Let $h$ be a real-valued random variable with finite moments, let $f$ be a $C^\infty(\mathbb{R})$ function. Then for any $\ell \in \mathbb{N}$ the following expansion holds,
        \begin{equation} \label{cumulant_formula}
            \E{h\cdot f(h)} = \sum\limits_{j=0}^\ell \frac{1}{j!}c^{(j+1)}(h)\E{\frac{\mathrm{d}^j}{\mathrm{d}h^j}f(h)} + R_{\ell+1},
        \end{equation}
        where $c^{(j)}$ is the $j$-th cumulant of $h$ defined by
        $$
        c^{(j)}(h) = (-i)^j\left.\frac{\mathrm{d}^j}{\mathrm{d}t^j}\left(\log\E{e^{ith}}\right)\right|_{t=0},
        $$
        and the remainder term $R_{\ell+1}$ satisfies
        \begin{equation} \label{cumulant_error}
           |R_{\ell+1}| \le C_l \E{|h|^{\ell+2}}\sup\limits_{|x|\le M}|f^{(\ell+1)}(x)| + C_l \E{|h|^{\ell+2}\cdot \mathds{1}_{|h|>M}}\norm{f^{(\ell+1)}(x)}_\infty,
        \end{equation}
        for any $M>0$.
    \end{lemma}
	We apply formula \eqref{cumulant_formula} with $h$ equal to the matrix element $H_{jk}$. Correspondingly, in the real case ($\beta=1$), $\Cmlnt^{(p)}$ denotes the matrix of $p$-th cumulants of $H$, $\Cmlnt^{(p)}_{jk} := \Cmlnt^{(p)}(H_{jk})$. In the complex case ($\beta=2$), $\Cmlnt^{(p)}$ is used as a notational shortcut and denotes the sum of matrices of $p$-th cumulants of real and imaginary parts of $H$, that is $\Cmlnt^{(p)}_{jk}:= \Cmlnt^{(p)}(\re H_{jk}) + \Cmlnt^{(p)}(\im H_{jk})$.

\section{Proof of the Main Result} \label{proof_sect}
	\begin{proof}[Proof of Theorem \ref{main_Th}]
		We divide the proof into two parts contained in the following propositions. We indicate their analogs in the settings of \cite{Li2021genWigner} and \cite{Landon2021Wignertype} in parenthesis. 	
		\begin{prop} (c.f. Theorem 2.2 in \cite{Li2021genWigner} and (5.76) in \cite{Landon2021Wignertype}) \label{main1}
			Let $\eta_0$, $\varepsilon_0 > 0$ and $E_0$ satisfy the assumptions of Theorem \ref{main_Th}, let $f$ be a scaled test function defined in \eqref{scaled_f}, 
			and	let $\phi(\lambda)$ be the characteristic function of $\Tr f(H) -\E{\Tr f(H)}$,
			\begin{equation} \label{phi_def}
				\phi(\lambda) := \E{\exp\{\I\lambda\left(\Tr f(H) - \E{\Tr f(H)}\right)  \}},\quad \lambda \in\mathbb{R}.
			\end{equation}
			Then its derivative $\phi'(\lambda)$ satisfies the following equation,
			\begin{equation} \label{main1_goal}
				\phi'(\lambda) = -\lambda\phi(\lambda)V(f) + \Oprec\bigl(N^{-1/2}\eta_0^{-1/2}(1+|\lambda|^4)+(1+|\lambda|)N^{-\alp}\bigr), \quad \lambda\in\mathbb{R},
			\end{equation}
			provided $c \le V(f) \le C$ for some positive $N$-independent constants $c$ and $C$.
			
			Here the variance $V(f)$ for a scaled test function $f$ is defined by
			\begin{equation} \label{variance_V}
				V(f) := \frac{1}{\pi^2}\int\limits_{\dom}\int\limits_{\dom'} \frac{\partial \other f(\zeta)}{\partial \bar \zeta}\frac{\partial \other f(z)}{\partial \bar z} \mathcal{K}(z,\zeta)\mathrm{d}\bar\zeta\mathrm{d}\zeta \mathrm{d}\bar z\mathrm{d}z,
			\end{equation}
			where for $z,\zeta \in \mathbb{C}/\mathbb{R}$ the kernel $\mathcal{K}(z,\zeta)$ is defined by
			\begin{equation} \label{kernel_K}
				\begin{split}
					\mathcal{K}(z,\zeta) :=& \frac{2}{\beta}\frac{\partial }{\partial\zeta}\Tr\biggl[\frac{\m'(z)}{\m(z)}\bigl(1-S\m(z)\m(\zeta)\bigr)^{-1}\biggr]
					\\
					&+\biggl(1-\frac{2}{\beta} \biggr) \Tr\left[S\m'(z)\m'(\zeta)\right] + \frac{1}{2}\frac{\partial^2}{\partial z\partial\zeta}\left\langle\overline{\m(z)\m(\zeta)}, \Cmlnt^{(4)}\m(z)\m(\zeta) \right\rangle,
				\end{split}
			\end{equation}
			with $\Cmlnt^{(4)}$ denoting the matrix of fourth cumulants $\Cmlnt^{(4)}_{jk}$. The integration domains $\dom, \dom'$ in \eqref{variance_V} are defined as
			\begin{equation} \label{Omega_0_defs}
				\dom := \{z\in \mathbb{C} : |\im{z}| > N^{-\alp}\eta_0 \},\quad	\dom' := \{z\in \mathbb{C} : |\im{z}| > 2N^{-\alp}\eta_0 \},
			\end{equation}
			and $\other{f}$ is the quasi-analytic extension of $f$, defined by
			\begin{equation} \label{QA_f}
				\other{f}(x+\I\eta) = \chi(\eta) \left( f(x) + \I\eta f'(x) \right),
			\end{equation}
			where $\chi :\mathbb{R} \to [0,1]$ is an even $C_c^\infty(\mathbb{R})$ function supported on  $[-1,1]$, satisfying $\chi(\eta) = 1$ for $|\eta|<1/2$.
		\end{prop}

		\begin{prop} (c.f. Lemma 6.7 in \cite{Landon2021Wignertype}) \label{main2}
			Let $E_0, \eta_0$ satisfy the conditions of Theorem \ref{main_Th}. 
			Let $f$ be the scaled test function with $g\in C_c^2(\mathbb{R})$ given in \eqref{scaled_f}, and let $V(f)$ be the variance defined in \eqref{variance_V}, then 
			\begin{equation}
				V(f) = \frac{1}{2\beta\pi^2}\norm{g}_{\dot{H}^{1/2}}^2
				 +\bigO{\eta_0\log N + N^{-\twoalp}}.
			\end{equation}
		\end{prop}
		Proposition \ref{main2} implies that $V(f)$ satisfies the condition of Proposition \ref{main1}, hence 
		\begin{equation}
			\phi'(\lambda) = -\lambda\phi(\lambda)V(f) + \smallo{1},
		\end{equation}
		as $N \to \infty$, for any fixed $\lambda \in \mathbb{R}$.
		It then follows by L\'evy's continuity theorem that $\Tr f(H) -\E{\Tr f(H)}$ converges in distribution to a centered Gaussian with variance $(2\beta\pi^2)^{-1}\norm{g}_{\dot{H}^{1/2}}^2$.
		Therefore, to establish Theorem \ref{main_Th}, it suffices to show that Propositions \ref{main1} and \ref{main2} hold, which is done in Sections \ref{proof_of_main1} and \ref{proof_of_main2}, respectively.
	\end{proof}

	\begin{remark}
		We restrict the proof to the real symmetric ($\beta =1$) matrices for the sake of presentation. The complex Hermitian ($\beta=2$) case differs solely in replacing the cumulant expansion formula (Lemma \ref{cumulant_lemma}) with its complex analog. The obvious modifications are left to the reader.
	\end{remark}

	\subsection{Characteristic Function of Linear Eigenvalue Statistics} \label{proof_of_main1}
	\begin{proof}[Proof of Proposition \ref{main1}]
	Using standard techniques of the characteristic function method imported from, e.g.,  Section 5.2 of   \cite{Landon2021Wignertype} (see also Section 4.2 of  \cite{Landon2020applCLT} and references therein), we can obtain the following series of estimates on the characteristic function of the linear eigenvalue statistics $\phi(\lambda)$ and its derivative $\phi'(\lambda)$.  The proof is a relatively straightforward modification of similar arguments in \cite{Landon2021Wignertype}, so we defer it to \ref{App1}.  
	\begin{lemma} \label{standard_estimate_lemma}
		Let $\phi(\lambda)$ be the characteristic function defined in \eqref{phi_def}, then, under the conditions of Theorem \ref{main_Th}, the following estimates hold 
		\begin{equation} \label{phi'_Omega_int}
			\begin{split}
				\phi(\lambda) &= \E{\other{e}(\lambda)} + \Oprec\bigl(N^{-\alp}\bigr),\\
				\phi'(\lambda) &= \frac{\I}{\pi}\int\limits_{\dom} \frac{\partial\other{f}}{\partial\bar{z}}
				\E{\other e(\lambda) \left\{1-\Expv\right\}\left[ \Tr G(z)\right] }\mathrm{d}\bar z \mathrm{d}z + \Oprec\bigl(|\lambda|N^{-\alp}\bigr),
			\end{split}
		\end{equation}
		where
		\begin{equation} \label{tilde e}
			\other{e}(\lambda) := \exp\biggl\{\frac{\I\lambda}{\pi} \int\limits_{\dom'}\frac{\partial \other{f}}{\partial \bar{z}}\{1-\Expv\}\left[\Tr G(z)\right] \mathrm{d}\bar z \mathrm{d}z \biggr\}.
		\end{equation}
		Furthermore, for all $z \in \mathcal{D}_\kappa$, we have
		\begin{equation} \label{trace_fluctuation}
			\begin{split} 
				\E{\other e(\lambda)\left\{1-\Expv\right\}\left[ \Tr G(z)\right]}
				=& \E{\other e(\lambda)\left\{1-\Expv\right\}\mathcal{T}(z,z)}
				+\frac{2\I\lambda}{\pi} \Expv\biggl[\other e(\lambda) \int\limits_{\dom'} \frac{\partial\other{f}}{\partial\bar\zeta} \frac{\partial}{\partial\zeta}\mathcal{T}(z,\zeta) \mathrm{d}\bar \zeta \mathrm{d}\zeta \biggr] \\
				&+\frac{\I\lambda}{\pi} \E{\other e(\lambda)} \int\limits_{\dom'} \frac{\partial\other{f}}{\partial\bar\zeta} \Tr\left[S \m'(z)\m'(\zeta)\right] \mathrm{d}\bar \zeta \mathrm{d}\zeta\\
				&+\frac{\I\lambda}{2\pi}\E{\other e(\lambda) }\int\limits_{\dom'}\frac{\partial\other{f}}{\partial\bar\zeta}\frac{\partial^2 }{\partial  z\partial \zeta}\bigl\langle \overline{\m(z)\m(\zeta)}, \Cmlnt^{(4)}\m(z)\m(\zeta)\bigr\rangle \mathrm{d}\bar \zeta \mathrm{d}\zeta\\
				&+\Oprec\bigl((1+|\lambda|^4)(N\Psi(z)\Theta(z) + \Psi(z)\eta_0^{-1/2})\bigr),
			\end{split}
		\end{equation} 
		where the random function $\mathcal{T}(z,\zeta)$ is defined as
		\begin{equation} \label{Curly T def}
			\mathcal{T}(z,\zeta) :=  \Tr\biggl[\frac{\m'(z)}{\m(z)} T(z,\zeta)\biggr]. 
		\end{equation}
	\end{lemma}
	 We now proceed to estimate the first two terms on the right-hand side of \eqref{trace_fluctuation} in such a way that $\E{\other{e}(\lambda)}$ factors out. By definition of the scaled test function \eqref{scaled_f}, the support of $\other{f}$ is contained inside a vertical strip centered at $E_0$ of width $\sim \eta_0$, hence we limit the further analysis to the regime $|\re\zeta - \re z| \lesssim \eta_0 \ll \epsilon$, where $\epsilon$ is defined in the statement of Theorem \ref{th_T_local_law}. 
	 We estimate the function $\mathcal{T}(z,\zeta)$ using Corollary \ref{cor_T_local_law} with weight matrix $A := \frac{\m'(z)}{\m(z)}$. It follows from the bounds \eqref{m_bound} and \eqref{dm_bound} that $\opnorminf{A} \lesssim 1$, hence for all $z,\zeta \in \mathcal{D}_\kappa$ with $\re z, \re\zeta \in \supp{f}$, 
	\begin{equation} \label{Curly T eq}
		\mathcal{T}(z,\zeta) = \Tr\biggl[ \frac{\m'(z)}{\m(z)} \bigl(1 - S\m(z)\m(\zeta) \bigr)^{-1} \bigl(S\m(z)\m(\zeta) \bigr)^2\biggr] + \mathcal{E}(z,\zeta),
	\end{equation}
	where the error term $\mathcal{E}(z,\zeta)$ is analytic in both variables and admits the bound
	\begin{equation} \label{Curly T error bounds}
		\begin{split}
			\mathcal{E}(z,\zeta) \prec N\Psi^2(z) \Psi(\zeta) + N\Psi(z)\Psi^2(\zeta) + \mathds{1}_{\{\im z\im \zeta <0\}}N\Theta(z)\Theta(\zeta)
		\end{split}.
	\end{equation} 	      
	
	 It follows from \eqref{Curly T eq} and \eqref{Curly T error bounds} for $\zeta = z$ that
	\begin{equation} \label{T(z,z)_error}
		\E{\other{e}(\lambda) \{1-\Expv\}\left[\mathcal{T}(z,z)\right]} \prec N\Psi(z)^3,
	\end{equation}
	yielding the desired bound on the first term on the right-hand side of \eqref{trace_fluctuation}.
	
	We now estimate the second term in \eqref{trace_fluctuation}. Fix $z \in \mathcal{D}_\kappa$, and consider $\zeta$ that lie in $\dom'$ defined in \eqref{Omega_0_defs}. Differentiating \eqref{Curly T eq} with respect to $\zeta$ yields
	\begin{equation}
		\frac{\partial}{\partial\zeta}\mathcal{T}(z,\zeta) = \frac{\partial}{\partial\zeta}\Tr\biggl[\frac{\m'(z)}{\m(z)} \bigl(1 - S\m(z)\m(\zeta) \bigr)^{-1} \bigl(S\m(z)\m(\zeta) \bigr)^2\biggr] + \frac{\partial}{\partial\zeta}\mathcal{E}(z,\zeta).
	\end{equation}
	To bound the derivative of the error term $\mathcal{E}(z,\zeta)$, we use the following technical lemma.
	\begin{lemma} (Lemma 5.5 in \cite{Landon2021Wignertype}) \label{derivative_lemma}
		Let $K(z)$ be a holomorphic function on $\mathbb{C}\backslash\mathbb{R}$,
		then for all $z \in \mathbb{C}\backslash\mathbb{R}$ and any $p\in \mathbb{N}$,
		\begin{equation}
			\left| \frac{\partial^p K}{\partial z^p}(z) \right| \le C_p |\im z|^{-p} \sup\limits_{|\zeta - z| \le |\im z|/2}|K(\zeta)|,
		\end{equation}
		where $C_p>0$ is a constant depending only on $p$.
	\end{lemma}
	Lemma \ref{derivative_lemma} applied to the estimate \eqref{Curly T error bounds} implies that the error term $\partial_\zeta\mathcal{E}(z,\zeta)$ admits the bound
	\begin{equation} \label{E'_error}
		\begin{split}
			\frac{\partial}{\partial\zeta}\mathcal{E}(z,\zeta) &\prec N|\im \zeta|^{-1}\bigl(\Psi(z)^2\Psi(\zeta) + \Psi(z)\Psi(\zeta)^2 +\Theta(z)\Theta(\zeta)\bigr). 
		\end{split}
	\end{equation}
	To proceed we require another technical lemma.
	\begin{lemma} (c.f. Lemma 4.4 in \cite{Landon2020applCLT}) \label{int_lemma}
		Let $f$ be the scaled test function defined in \eqref{scaled_f}. Let $\Omega$ be a domain of the form
		\begin{equation}
			\Omega := \{z\in\mathbb{C}: cN^{-\tau'}\eta_0 < |\im z| < 1, a < \re z  < b\},
		\end{equation} 
		such that $\supp{f} \subset (a,b)$ and $\tau',c$ are positive constants.
		Let $K(z)$ be a holomorphic function on $\Omega$ satisfying
		\begin{equation}
			|K(z)| \le C|\im z|^{-s},\quad z \in \Omega, 
		\end{equation}
		for some $0\le s \le 2$.
		Then there exists a constant $C' > 0$ depending only on $g$ in \eqref{scaled_f}, $\chi$ in \eqref{QA_f}, and $s$, such that
		\begin{equation} \label{int_lemma_bound}
			\biggl|\int\limits_{\Omega}\frac{ \partial\other{f}}{\partial\bar z}(x+\I y) K(x+\I y)\mathrm{d}x\mathrm{d}y \biggr| \le C C'  \eta_0^{1-s}\log N.
		\end{equation}
	\end{lemma}
	\begin{proof}[Proof of Lemma \ref{int_lemma}]
		It follows from \eqref{scaled_f} that $\norm{f}_1 \sim \eta_0, \norm{f'}_1 \sim 1, \norm{f''}_1 \sim \eta_0^{-1}$.
		In case $1\le s\le 2$ the inequality \eqref{int_lemma_bound} follows from Lemma 4.4 in \cite{Landon2020applCLT}.
		For $0\le s < 1$, the proof is conducted along the same lines, except the integration by parts is performed twice in the regime $\eta_0 \le |\im z| \le 1$.
	\end{proof}	
	
	Lemma \ref{int_lemma} and the matrix identity $(1-X)^{-1}X^2 = (1-X)^{-1} - X - 1$ yield the following expression.
	\begin{equation} \label{partial_estimate}
		\begin{split}
			\Expv\biggl[\other{e}(\lambda)\int\limits_{\dom'} \frac{\partial\other{f}}{\partial\bar \zeta}\frac{\partial\mathcal{T}}{\partial\zeta}\mathrm{d}\bar \zeta \mathrm{d}\zeta\biggr] 
			=& \E{\other{e}(\lambda)}\int\limits_{\dom'} \frac{\partial\other{f}}{\partial\bar \zeta}\frac{\partial}{\partial\zeta}\Tr\biggl[\frac{\m'(z)}{\m(z)}\bigl(1 - S\m(z)\m(\zeta) \bigr)^{-1}\biggr]  \mathrm{d}\bar \zeta \mathrm{d}\zeta \\
			-&\E{\other{e}(\lambda)}\int\limits_{\dom'} \frac{\partial\other{f}}{\partial\bar \zeta}\Tr\bigl[S\m'(z)\m'(\zeta) \bigr]  \mathrm{d}\bar \zeta \mathrm{d}\zeta \\
			+& \Oprec\bigl(N^{1/2}\Psi(z)^2\eta_0^{-1/2} + \Psi(z)\eta_0^{-1} +\Theta(z)\eta_0^{-1}\bigr),
		\end{split}
	\end{equation}
	Finally, from \eqref{trace_fluctuation} and \eqref{partial_estimate}, combined with \eqref{phi'_Omega_int} we conclude that
	\begin{equation} \label{phi'_expr_V_tilde_e}
		\phi'(\lambda) = -\lambda V(f) \E{\other e(\lambda)} + \other\Epsilon(\lambda),
	\end{equation}
	where $V(f)$ is defined in \eqref{variance_V}, and $\other\Epsilon(\lambda)$ is the total error term collected from previous derivations and integrated over $\mathrm{d}\bar{z}\mathrm{d}z$.
	Lemma \ref{int_lemma} together with error estimates in \eqref{phi'_Omega_int}, \eqref{trace_fluctuation}, \eqref{T(z,z)_error} and \eqref{E'_error} provides the following bound on the error term
	\begin{equation}
		\begin{split}
			\other{\mathcal{E}} 
			&= \Oprec\bigl(N^{-1/2}\eta_0^{-1/2}(1+|\lambda|^4)+|\lambda|N^{-\alp}\bigr).
		\end{split}
	\end{equation}
	Under the conditions of Proposition \ref{main1} $V(f)$ is bounded, hence we conclude from the first estimate in \eqref{phi'_Omega_int} and  \eqref{phi'_expr_V_tilde_e} that \eqref{main1_goal} holds.
	This concludes the proof of Proposition \ref{main1}.
	\end{proof}
\section{Proof of the Local Laws for Two-point Functions} \label{Curly T section}
	In this section, we derive all the tools necessary to prove Theorem \ref{th_T_local_law} and its specification for the two-point function $T(z,\zeta)$, Corollary \ref{cor_T_local_law}. To make the notation more concise we introduce the convention $G \equiv G(z)$, $\other{G} \equiv G(\zeta)$, $\m \equiv \m(z)$, $\other\m \equiv \m(\zeta)$, $\other\Psi \equiv \Psi(\zeta)$, $\Psi \equiv \Psi(z)$, $\Theta \equiv \Theta(z)$, $\other\Theta \equiv \Theta(\zeta)$.
		
	For a deterministic matrix $W$ with entries $|W_{ab}| \lesssim N^{-1}$, the quantity $\sum_{a\neq y} W_{ax}G_{\alpha a}\other{G}_{a\beta}$ can be readily estimated in two special cases.
	First, if each column of $W$ is proportional to the vector of ones, i.e., $W_{ab} = w_b$ depends only on $b$, then the summation over $a$ yields $w_x ([G\other{G}]_{\alpha\beta} - G_{\alpha y}\other{G}_{y\beta})$, and the estimate follows from the resolvent identity and the local laws in Theorem \ref{local_law}.
	Second, if the entries of $X:=(1-S\m\other\m)^{-1}W$ are bounded by $CN^{-1}$, then one can obtain the estimate from Lemma \ref{self_eq_lemma} below.  We show that these two special cases are exhaustive in the sense that any $W$ can be represented as their linear combination with controlled coefficients. 
	
	To this end, we prove that in the relevant regime, the operator $(1-S\m\other\m)$ has a very small destabilizing eigenvalue and an order one spectral gap above it. Moreover, if $\Pi$ is the eigenprojector corresponding to the principal eigenvalue of $(1-S\m\other\m)$, then the $\ell^\infty \to \ell^\infty$-norm of the restriction of $(1-S\m\other\m)^{-1}$ to the kernel of $\Pi$ is also an order one quantity. 
	Finally, we show that the vector of ones $\vect{1}$ is sufficiently separated from the kernel of $\Pi$.	
	\subsection{Stable Direction Local Law} 
	 For any $N\times N$ deterministic matrix $W$, and any indices $x,y,\alpha,\beta$, we define the quantities 
	\begin{equation} \label{F_quant_def}
		\Ff{xy}{\alpha\beta}(W) := \sum_{a\neq y}W_{ax} G_{\alpha a}\other{G}_{a\beta}, \quad \mathpzc{f}_{\alpha}^{xy}(W) := m_\alpha \other{m}_\alpha ([(1-S\m\other{\m})^{-1}W]_{\alpha x} -\delta_{\alpha y}W_{\alpha x}).
	\end{equation}  
	We prove the following estimate.
	%
	\begin{lemma} \label{self_eq_lemma}
		For any $z, \zeta \in \mathcal{D}_\kappa$ and any deterministic $N\times N$ matrix $\X$,
		\begin{equation} \label{(1-Smm)T_law}
			 \Ff{xy}{\alpha\beta}((1-S\m\other{\m})\X) = \delta_{\alpha\beta}\mathpzc{f}_{\alpha}^{xy}((1-S\m\other{\m})\X) + \Oprec\bigl(N\norm{X}_{\max}\Psi\other{\Psi}(\Psi+\other{\Psi})
			\bigr). 
		\end{equation}
	
		provided $\norm{\X}_{\max} := \max\limits_{j,k} |\X_{jk}| \lesssim 1$.
	\end{lemma}
	
	We use the following self-improving mechanism for stochastic domination bounds, borrowed, e.g., from \cite{He2019Diffusion}.
	\begin{lemma} (Lemma 6.3 in \cite{He2019Diffusion}) \label{lemma_self_improv}
		Let $\mathcal{X}$ be a random variable such that $0 \le \mathcal{X} \prec N^C$ for some $C>0$, and let $\Xi \ge 0$ be a deterministic quantity. Suppose there exists a constant $q\in [0,1)$, such that for any $\Phi$ satisfying $ \Xi \le \Phi \le N^C$, and any $d\in\mathbb{N}$, we have the implication 
		\begin{equation}
			\mathcal{X} \prec \Phi\quad \Longrightarrow \quad \Expv\bigr[|\mathcal{X}|^{2d}\bigl] \prec \sum\limits_{k=1}^{2d}\bigl(\Phi^q\Xi^{1-q})^k\Expv\bigr[|\mathcal{X}|^{2d-k}\bigl],
		\end{equation}
		then $\mathcal{X} \prec \Xi$.
	\end{lemma}

	\begin{proof}[Proof of Lemma \ref{self_eq_lemma}]
		Let $\Y := (1-S\m\other{\m})\X$, then the quantity we need to estimate is $[\G \Y]_{yx} = \Ff{xy}{yy}(\Y)$. It follows from the local law in the form \eqref{law1} that 
		\begin{equation} \label{control_Lambda}
			\Ff{xy}{\alpha\beta}(\Y) \prec N\norm{X}_{\max}\Psi\other{\Psi} =: \Lambda.
		\end{equation}
		Let $\Phi$ be a deterministic control parameter admitting the bounds $(\Psi+\other{\Psi})\Lambda\le \Phi\le\Lambda$, such that 
		\begin{equation} \label{control_Phi}
			\Ff{xy}{\alpha\beta}(\Y) - \delta_{\alpha\beta}\mathpzc{f}_{\alpha}^{xy}(\Y) \prec \Phi.
		\end{equation}
		It follows trivially from \eqref{control_Lambda} and \eqref{control_Phi} that 
		\begin{equation} \label{control_trivial}
			\Ff{xy}{\alpha\beta}(\Y) \prec \Phi + \delta_{\alpha\beta}\Lambda.
		\end{equation} 
		Let $\partial_{jk}$ denote the partial derivative with respect to the matrix element $H_{jk}$, then the partial derivatives of $\Ff{xy}{\alpha \beta}$ are given by
		\begin{equation} \label{Ff_deriv}
			\partial_{ab}\Ff{xy}{\alpha\beta}(\Y) = -(1 + \delta_{ab})^{-1}(G_{\alpha a}\Ff{xy}{b\beta}(\Y) + G_{\alpha b}\Ff{xy}{a\beta}(\Y) +\Ff{xy}{\alpha b}(\Y) \other{G}_{a\beta} + \Ff{xy}{\alpha a}(\Y)\other{G}_{b\beta}).	
		\end{equation}
		
		We combine the vector Dyson equation \eqref{VD_eq} and the resolvent identity $zG = HG - 1$ to obtain
		\begin{equation} \label{G_entry}
			\other{G}_{a\beta} = -\other{m}_a \sum\limits_b \left( H_{ab}\other{G}_{b\beta} + S_{ab} \other{m}_b \other{G}_{a\beta} \right) + \other{m}_a\delta_{a\beta}.
		\end{equation}
		Let $d\in \mathbb{N}$, define $\mathpzc{P} \equiv \mathpzc{P}(d-1,d) := (\Ff{xy}{\alpha\beta}(\Y) - \delta_{\alpha\beta}\mathpzc{f}_{\alpha}^{xy}(\Y))^{d-1}(\overline{\Ff{xy}{\alpha\beta}(\Y) - \delta_{\alpha\beta}\mathpzc{f}_{\alpha}^{xy}(\Y)})^d$. For any $p\in \mathbb{N}$, define $M_p := \Expv\bigl[|\Ff{xy}{\alpha\beta}(\Y) - \delta_{\alpha\beta}\mathpzc{f}_{\alpha}^{xy}(\Y)|^p\bigr]$.
		Plugging \eqref{G_entry} into the definition \eqref{F_quant_def} and applying the cumulant expansion formula of Lemma \ref{cumulant_lemma}, we obtain
		\begin{subequations} \label{general_cumul_expan}
			\begin{align} 
				\Expv\bigl[\Ff{xy}{\alpha\beta}(\X)\mathpzc{P}\bigr] =
				&
				\sum_{a \neq y} m_a\other{m}_a\X_{ax} \Expv\bigl[\Ff{ay}{\alpha \beta}(S)\mathpzc{P}\bigr] 
				+\delta_{\alpha \beta} \mathpzc{f}_{\alpha}^{xy}(\Y)\Expv[\mathpzc{P}] + \delta_{\alpha \beta} \delta_{\beta y}S_{yy}m_y^2\other{m}_y^2 \X_{yx}\Expv[\mathpzc{P}] \label{expan_1}\\
				&
				+ \Expv\bigl[\sum_{a\neq y}\sum_{b} \other{m}_a \X_{ax}S_{ab} \bigl( G_{\alpha a}(\other{G}_{bb} - \other{m}_b)\other{G}_{a\beta}+ G_{\alpha b}(G_{aa} - m_a)\other{G}_{b\beta}\bigr)\mathpzc{P}\bigr] \label{expan_2}\\
				&
				+ \Expv\bigl[\sum_{a \neq y}\sum_{b\neq a} \other{m}_a \X_{ax} S_{ab} G_{\alpha a}\bigl(G_{ba} + \other{G}_{ba}\bigr)\other{G}_{b\beta}\mathpzc{P}\bigr] \label{expan_3} + R_2\\
				&
				+\sum_{a\neq y}\X_{ax}m_a\other{m}_aS_{ay}\Expv\bigl[\bigl(G_{\alpha y}\other{G}_{y\beta} - \delta_{\alpha y}\delta_{y\beta}m_y\other{m}_y\bigr)\mathpzc{P}\bigr] \label{expan_4}\\
				&
				+ \delta_{\beta \neq y}\other{m}_\beta \X_{\beta x}\Expv\bigl[(G_{\alpha\beta}-\delta_{\alpha\beta}m_\beta)\mathpzc{P}\bigr] 
				- \Expv\bigl[\sum_{a\neq y}\other{m}_a \X_{ax}G_{\alpha a}\sum_{b}S_{ab}\other{G}_{b\beta}\partial_{ab}\mathpzc{P}\bigr], \label{expan_5}
			\end{align}
		\end{subequations}
		where $R_2$ is the total error coming from the higher order cumulants, and all unrestricted summations are from $1$ to $N$.
		We successively bound the terms \eqref{expan_2}-\eqref{expan_5} appearing on the right-hand side of \eqref{general_cumul_expan}. By condition \eqref{cond_A}, local law \eqref{law1}, upper bound \eqref{m_bound}, and \eqref{control_Phi}, it follows that the terms \eqref{expan_2} and the first term in \eqref{expan_3} are bounded by $\Oprec((\Psi+\other{\Psi})\Lambda M_{2d-1})$. Similarly, the term \eqref{expan_4} and the first term in \eqref{expan_5} are bounded by $\Oprec(\norm{X}_{\max}(\Psi+\other{\Psi})M_{2d-1})$.
		
		We bound the second term in \eqref{expan_5}. It follows by \eqref{cond_A}, \eqref{law1}, bounds \eqref{m_bound}, \eqref{control_trivial}, and \eqref{Ff_deriv} that 
		\begin{equation} \label{sum_deriv_bound}
			\sum_{b}S_{ab}\other{G}_{b\beta}\partial_{ab}\mathpzc{P} \prec (\Psi +\other{\Psi} + \delta_{\alpha a} + \delta_{a\beta})\other{\Psi}\Phi M_{2d-2}.
		\end{equation}
		Hence, the second term in \eqref{expan_5} is bounded by $\Oprec\bigl((\Psi+\other{\Psi})\Lambda \Phi M_{2d-2}
		\bigr)$. 
		Finally, it is easy to check using estimates \eqref{cumulant_error}, \eqref{control_trivial} and identity \eqref{Ff_deriv}, together with condition \eqref{cond_A} and \eqref{m_bound}, that the error term $R_2 \prec (\Psi+\other{\Psi})\Lambda M_{2d-1} + (\Psi+\other{\Psi})\Lambda\Phi M_{2d-2} + (\Psi + \other{\Psi})\Lambda \Phi^2 M_{2d-3}$.
		
		 Observe that the first term on the right-hand side of \eqref{expan_1} can be expressed as
		\begin{equation} \label{expan_cancellation}
			\sum_{a \neq y} m_a\other{m}_a\X_{ax} \Expv\bigl[\Ff{ay}{\alpha \beta}(S)\mathpzc{P}\bigr]  = \Expv\bigl[\Ff{ay}{\alpha \beta}(\X)\mathpzc{P}\bigr] - \Expv\bigl[\Ff{ay}{\alpha \beta}(\Y)\mathpzc{P}\bigr]  - m_y\other{m}_y\X_{yx} \Expv\bigl[\Ff{yy}{\alpha \beta}(S)\mathpzc{P}\bigr], 
		\end{equation}
		where the last term is bounded by $\Oprec(N^{-1}\Lambda M_{2d-1})$. Combining \eqref{general_cumul_expan} and \eqref{expan_cancellation} yields 
		\begin{equation}
			\Expv\bigl[|\Ff{xy}{\alpha\beta}(\Y) - \delta_{\alpha\beta}\mathpzc{f}_{\alpha}^{xy}(\Y)|^{2d}\bigr] \prec \bigl(\Psi+\other{\Psi}
			\bigr)\Lambda\Phi
			^{2}M_{2d-3},
		\end{equation}
		for any control parameter $\Phi_{\alpha\beta,y}$ satisfying \eqref{control_Phi}. Hence, by Lemma \ref{lemma_self_improv},
		\begin{equation} \label{Ff_local_law}
			\Ff{xy}{\alpha\beta}(\Y) = \delta_{\alpha\beta}\mathpzc{f}_{\alpha}^{xy}(\Y) + \Oprec\bigl(\Lambda(\Psi+\other{\Psi})
			\bigr), 
		\end{equation}
		which concludes the proof of Lemma \ref{self_eq_lemma}.
	\end{proof}
\begin{remark} \label{same_half_plane_remark}
	 If $z$ and $\zeta$ are in the same (upper or lower) half-plane, Lemma \ref{self_eq_lemma} implies Theorem \ref{th_T_local_law}. Indeed, the bound \eqref{norm in same half plane} in Proposition \ref{Smm_prop} shows that provided $\eta\other{\eta} >0$, $X := (1-S\m\other{\m})^{-1}W$ satisfies $|X_{jk}| \lesssim N^{-1}$. Applying Lemma \ref{self_eq_lemma} to $X = (1-S\m\other{\m})^{-1}W$ then yields \eqref{twoG_law_1}, and \eqref{twoG_law_2} follows by summing \eqref{twoG_law_1}. We turn to the case of $z$ and $\zeta$ lying in different (upper and lower) half-planes.
\end{remark}
%
%
	%
	%
	\subsection{Stability Operator Analysis} \label{Stab_subsection}
	In this subsection we obtain all the properties of the stability operator $\left(1-S\m(z)\m(\zeta)\right)$ that we use in combination with Lemma \ref{self_eq_lemma} to finish the proof of Theorem \ref{th_T_local_law} for $z,\zeta$ lying in opposite half-planes, as outlined in the beginning of Section \ref{Curly T section}. 
	
	For two spectral parameters $z,\zeta$, let $\eta := \im z$, and $\other{\eta} := \im \zeta$.	
	Without loss of generality, we assume in the following that $\re z \in \bulk_\kappa, \eta>0$ and $\re\zeta \in \bulk_\kappa, \other\eta<0$.
	For the remainder of this subsection, we use the following notation
	\begin{equation}
		\begin{split} \label{F_A_A_0_notation}
			F \equiv F(z) &:= |\m(z)|S|\m(z)|,\\
			\stab \equiv \stab(z,\zeta) &:= 1 - S\m(z)\m(\zeta),\\
			\stab_0 \equiv \stab_0(z) &:= 1 - S|\m(z)|^2 = |\m(z)|^{-1}(1 - F) |\m(z)|.
		\end{split}
	\end{equation}
	We view the operator $\stab$ as a perturbation of $\stab_0 = \stab(z,\bar z)$, since $|\zeta - \bar z|$ is small. We deduce the desired properties of $\stab$ from those of $\stab_0$, which, in turn, follow from the lower bound on the spectral gap of $F$ found in \eqref{gapF}.
	
	Let $\{\eigF_j\}_{j=1}^N$ denote the eigenvalues of $F$ (with multiplicity) in descending order. Then, by Perron–Frobenius theorem, the principal eigenvalue $\eigF_1$ is real, and it coincides with the spectral radius $\opnormtwo{F}$.
	Furthermore, by taking the imaginary part of the vector Dyson equation \eqref{VD_eq} and multiplying both sides by $|\m|$ coordinate-wise, we obtain
	\begin{equation} \label{1-F w}
		\bigl(1-F\bigr)\frac{\im\m}{|\m|} = \eta|\m|.
	\end{equation}

	Furthermore, by condition \eqref{cond_A}, for every $j$ we have $(S \im \m)_j \sim  \langle\im\m(z)\rangle \sim\rho(z)$, where $\rho(z)$ is the harmonic extension of the self-consistent density of states $\rho(x)$ defined in \eqref{rho_def} into $\mathbb{C}$. Hence by taking the imaginary part of \eqref{VD_eq}, we get
	\begin{equation} \label{|m| im m comparison}
		\frac{\im m_j}{|m_j|} \sim |m_j|(\rho(z) + \eta), ,\quad j\in \{1,\dots, N\}.
	\end{equation}
	Therefore, by \eqref{1-F w} and \eqref{|m| im m comparison},  $1 - \eigF_1 \lesssim \eta$.  Together with an upper bound \eqref{F_norm} on $\opnormtwo{F}$, this implies that $1 - \eigF_1 \sim \eta$.
	It follows from \eqref{gapF}  that the principal eigenvalue of $F$ is separated from the rest of the spectrum by an annulus, i.e., there exist $r>0$ and $\delta>0$ independent of $z$ and $N$ such that 
	\begin{equation} \label{F separation}
		|1 - \eigF_1| < r-\delta,\quad \text{and}\quad |1-\eigF_j|>r+\delta, \quad j \in \{2,\dots, N\}.
	\end{equation}
	
	In the remainder of this subsection, we show that for all $\zeta$ sufficiently close to $\bar z$, the eigenvalue of $\stab$ with the smallest modulus is also separated from the rest of the spectrum by an annulus of order one width.
	
	Using the argument principle and Jacobi's formula, one can express the number of eigenvalues (with multiplicity) of a matrix $X$ inside a domain $\Omega$ by a contour integral
	\begin{equation} \label{NX Omega}
		N_X(\Omega) = \frac{1}{2\pi i} \oint\limits_{\partial\Omega} \Tr (w -X)^{-1} \mathrm{d}w.
	\end{equation}
	To show the eigenvalue separation for $\stab$, we begin by estimating the norm of the resolvent of $\stab$ inside the annulus
	\begin{equation} \label{annulus_def}
		\annulus_{r,\delta} := \{w\in \mathbb{C}: r-3\delta/4 \le |w| \le r+3\delta/4 \},
	\end{equation}
	with $r$ and $\delta$ as in \eqref{F separation}.
	\begin{claim} \label{A resolvent bound}
		There exists $\varepsilon_1 > 0$ and $\other{C} >0$ independent of $N$ and $z$ such that 
		\begin{equation}
			\norm{\left(w-\stab(z,\zeta)\right)^{-1}} \le \other{C}
		\end{equation}
		holds for all $w \in \annulus_{r,\delta}$ and all $\zeta$ such that $\re\zeta\in\bulk_\kappa,\im\zeta <0$ and $|\zeta-\bar z| \le \varepsilon_1$. (The norm $\norm{\cdot}$ is induced by either $\ell^2$ or $\ell^\infty$.)
	\end{claim}
	\begin{proof}
		Observe that $\norm{(w-\stab)^{-1}} \le \norm{ \bigl(1 - (w-\stab_0)^{-1}(\stab-\stab_0) \bigr)^{-1}}\norm{(w-\stab_0)^{-1}}$.\\
		Since $(w-\stab_0)^{-1} = -|\m|^{-1}(1-w - F)^{-1}|\m|$ and $|\m| \sim 1$, \eqref{F separation} implies that
		\begin{equation} \label{A0 resolvent bound}
			\norm{(w-\stab_0)^{-1}} \le \frac{C}{\min\limits_j|\eigF_j - w|} \le \frac{4C}{\delta},\quad w \in \annulus_{r,\delta}.
		\end{equation}
		From the uniform bounds \eqref{m_bound}, \eqref{dm_bound} on $|\m|$ and $|\m'|$ we have $\norm{\stab-\stab_0} \lesssim |\zeta-\bar z|$, which implies that there exists $\varepsilon_1 >0$ such that 
		\begin{equation} \label{perturb_norm_bound}
			\forall \zeta : |\zeta - \bar z| \le \varepsilon_1, \norm{\stab-\stab_0} \le \frac{\delta}{8C},
		\end{equation}
		where $C$ is the constant in \eqref{A0 resolvent bound}.\\
		It follows immediately that $\norm{ \bigl(1 - (w-\stab_0)^{-1}(\stab-\stab_0) \bigr)^{-1}} \le 2$ and hence 
		\begin{equation} \label{A resolvent bound 8C delta}
			\norm{(w-\stab)^{-1}} \le \frac{8C}{\delta}. \qedhere
		\end{equation}
	\end{proof}
	
	Claim \ref{A resolvent bound} implies that for any sufficiently large fixed $N$ the integrand in \eqref{NX Omega} with $X := \stab$ is uniformly bounded in $\Omega := \annulus_{r,\delta}$ for all $\zeta$ such that $|\zeta - \bar z| \le \varepsilon_1$, hence by analyticity 
	\begin{equation} \label{no_eigs_in_annulus}
		N_{\stab(z,\zeta)}(\annulus_{r,\delta}) = 0, \quad |\zeta - \bar z| \le \varepsilon_1.
	\end{equation}
	Since the eigenvalues of $\stab(z,\zeta)$ are continuous in $\zeta$, \eqref{no_eigs_in_annulus} implies that no eigenvalue can move between the two connected components of $\mathbb{C}\backslash \annulus_{r,\delta}$, which together with \eqref{F separation} yields the following claim.
	\begin{claim} \label{A separation}
		For any sufficiently large $N$, the equalities
		\begin{equation}
			\begin{split}
				N_\stab(\{|w| < r-3\delta/4\}) &=  N_{\stab_0}(\{|w| < r-3\delta/4\}) = 1, \\
				N_\stab(\{|w| > r+3\delta/4\}) &= N_{\stab_0}(\{|w| > r+3\delta/4\}) = N-1,
			\end{split}
		\end{equation}
		hold for any $\zeta$ such that $\re\zeta\in\bulk_\kappa,\im\zeta<0$ and $|\zeta - \bar z| \le \varepsilon_1$.
	\end{claim}
	
	Claim \ref{A separation} now allows us to define the principal eigenprojector $\Pi$ of $\stab$ as a contour integral
	\begin{equation} \label{Pi}
		\Pi \equiv \Pi(z,\zeta) := \frac{1}{2\pi i} \oint\limits_{|\xi| = r}(\xi - \stab(z,\zeta))^{-1}\mathrm{d}\xi.
	\end{equation}
	Claim \ref{A separation} asserts that the contour $\{ |\xi| =r\}$ encircles exactly one eigenvalue of $\stab$ with multiplicity, hence $\Pi$ is a rank one eigenprojector.
	
	We now prove that the restriction of $\stab^{-1}$ to the range of $(1-\Pi)$ is bounded by a constant.
	\begin{claim} \label{(1-Smm)(1-Pi)}
		For all $z,\zeta$ such that $\re z,\re\zeta \in \bulk_\kappa$, $\im z\im\zeta <0$ and $|\zeta - \bar z| \le \varepsilon_1$,
		\begin{equation} 
			\opnorminf{\stab^{-1}(1-\Pi)} \le \other{c},
		\end{equation}
		where $\other{c}$ depends only on the constants in conditions \eqref{cond_A}, \eqref{cond_B} and $\kappa$.
	\end{claim}
	\begin{proof}
		By expression \eqref{Pi} for $\Pi$ we have
		\begin{equation}
			\begin{split}
				\stab^{-1}(1-\Pi) &= - \frac{1}{2\pi i} \oint\limits_{|\xi| = r}\frac{1}{\xi}(\xi - \stab)^{-1}\mathrm{d}\xi
			\end{split}
		\end{equation}
		Hence the norm of $\stab^{-1}(1-\Pi)$ is bounded by
		\begin{equation}
			\opnorminf{\stab^{-1}(1-\Pi)} \le \frac{1}{2\pi}\int\limits_0^{2\pi} \opnorminf{ \left(r e^{\I\theta} - \stab \right)^{-1}}\mathrm{d}\theta \le \frac{8C}{\delta},
		\end{equation}
		using the bound in Claim \ref{A resolvent bound} on the circle $\{|\xi| = r\}$ which lies inside $\annulus_{r,\delta}$.
	\end{proof}
	Finally, we show that the vector of ones is sufficiently separated from the kernel of $\Pi$. This ensures a stable decomposition of the space into the direct sum of the range of $(1-\Pi)$ and the span of ${\vect{1}}$, so we can apply the local laws to each of the components separately.
	\begin{claim} \label{Pi 1 > c}
		There exists $\varepsilon >0$ independent of $N$ and $z$ such that for all $\zeta$ with $\re\zeta\in\bulk_\kappa,\im\zeta<0$ and $|\zeta- \bar z| \le \varepsilon$,
		\begin{equation}
			\frac{\norm{\Pi\vect{1}}_\infty}{\opnorminf{\Pi}} \ge c,
		\end{equation}
		where $c>0$ is a constant independent of $N$ and $z$.
	\end{claim}
	\begin{proof}
		Define the projector $\Pi_0$ corresponding to $\stab_0$ via \eqref{Pi}. Then $\Pi_0 = |\m|^{-1}\other{\Pi}_0|\m|$, where $\other{\Pi}_0$ is the orthoprojector corresponding to the principal eigenvalue of the Hermitian operator $F$.\\
		Since $|\m| \sim 1$ we have $\opnorminf{\Pi_0} \le C_0$. Moreover, by Proposition \ref{F_prop}, the $\ell^2$-normalized eigenvector $\vv$ corresponding to the principal eigenvalue of $F$ has entries $\vv_j \ge 0$ with $\vv_j \sim N^{-1/2}$, hence 
		\begin{equation}
			\norm{\Pi_0\vect{1}}_\infty = \norm{|\m|^{-1}\other{\Pi}_0|\m|\vect{1}}_\infty  = \norm{|\m|^{-1}\vv}_\infty\left\langle \vv,|\m|\right\rangle \ge c_0,
		\end{equation}
		where $c_0 > 0$ is a constant independent of $N$ and $z$.
		
		Similarly to the proof of \eqref{perturb_norm_bound}, for any $\gamma \in (0,1]$ there exists $\varepsilon_\gamma > 0$, such that the bound
		\begin{equation}
			\opnorminf{\stab-\stab_0} \le \gamma \frac{\delta}{8C}
		\end{equation}
		holds for all $\zeta\in\mathcal{D}_\kappa^-$ with $|\zeta - \bar z| \le \varepsilon_\gamma$. Here $\delta$ is defined in \eqref{F separation} and $C>0$ is the constant in \eqref{A0 resolvent bound}.
		We choose $\varepsilon_\gamma$ to be smaller than $\varepsilon_1$ of Claim \ref{A resolvent bound}, then for all $\zeta$ with $\re\zeta\in\bulk_\kappa$, $\im\zeta<0$ such that $|\zeta - \bar z| \le \varepsilon_\gamma$ we have
		\begin{equation} \label{Pi-Pi_0}
			\begin{split}
				\opnorminf{\Pi-\Pi_0} &\le \frac{r}{2\pi}\int\limits_0^{2\pi} \opnorminf{(r e^{\I\theta} -\stab)^{-1} - (r e^{\I\theta} - \stab_0)^{-1}}\mathrm{d}\theta\\
				&\le \frac{r}{2\pi}\int\limits_0^{2\pi} \opnorminf{(r e^{\I\theta} -\stab)^{-1}(\stab-\stab_0)(r e^{\I\theta} - \stab_0)^{-1}}\mathrm{d}\theta\\
				&\le r\cdot \frac{8C}{\delta} \cdot\gamma\frac{\delta}{8C} \cdot \frac{4C}{\delta} = \gamma \frac{4Cr}{\delta}.
			\end{split}
		\end{equation}
		Here we used inequalities \eqref{A0 resolvent bound} and \eqref{A resolvent bound 8C delta} in the second to last step. 
		We set the value of $\gamma$ to be $\gamma_0 := \min\left\{1, \frac{c_0\delta}{8Cr} \right\}$, which guarantees that 
		\begin{equation}
			\norm{\Pi\vect{1}}_\infty \ge \bigl| \norm{\Pi_0\vect{1}}_\infty - \opnorminf{\Pi-\Pi_0}\norm{\vect{1}}_\infty  \bigr| \ge c_0 - \gamma_0 \frac{4Cr}{\delta} \ge \frac{c_0}{2}.
		\end{equation}
		Finally, observe that 
		\begin{equation} \label{norm_Pi}
			\opnorminf{\Pi} \le \opnorminf{\Pi_0} + \opnorminf{\Pi-\Pi_0} \le C_0 + c_0/2.
		\end{equation}
		This proves the claim with $c := c_0 / (2C_0 + c_0)$.
	\end{proof}
	
\subsection{Finishing the Proof of Theorem \ref{th_T_local_law}}
\begin{proof}[Proof of Theorem \ref{th_T_local_law}]
	 Recall that the objective is to estimate the quantities defined in \eqref{twoGdef}. 
	Instead of estimating $\sum_{a\neq y} w_aG_{\alpha a} \other{G}_{a\beta}$ directly, it is more convenient to work with objects of the type $\sum_{a\neq y} W_{ax} G_{\alpha a} \other{G}_{a\beta}$, since they generalize quantities appearing in both \eqref{twoG_law_1} and \eqref{twoG_law_2}. The redundant index $x$ can be eliminated by setting $W_{ax} := w_a$. 
	
	In the case $\im z \im\zeta > 0$, \eqref{twoG_law_1} and \eqref{twoG_law_2} follow immediately from \eqref{norm in same half plane} and Lemma \ref{self_eq_lemma} (see Remark \ref{same_half_plane_remark}). Therefore, we focus on the case $\im z\im\zeta < 0$.

	Since $\Pi$ has rank one and Claim \ref{Pi 1 > c} asserts that $\Pi\vect{1} \neq 0$, the kernel of $\Pi$ together with $\vect{1}$ span $\mathbb{C}^N$. Therefore we can decompose each column of the  matrix $W$  into a linear combination of $\vect{1}$ and an element of $\ker\Pi$, that is, there exists an $N\times N$ matrix $Y$ and a vector $\vect{s}\in\mathbb{C}^N$ such that
	\begin{equation} \label{decomp g}
		W =  \Y  + \vect{1}\s^*,  \quad \Pi\Y = 0.
	\end{equation}
	We multiply the first equality in \eqref{decomp g} by $\Pi$ from the left, apply both sides to the $a$-th standard basis vector $\vect{e}_a$ of $\mathbb{C}^N$ and take the $\ell^\infty$-norm to deduce
	\begin{equation} \label{abs_s_a}
		 \norm{\Pi W\vect{e}_a}_\infty  = |s_a|\norm{\Pi\vect{1}}_\infty,\quad a\in\{1,\dots,N\}.
	\end{equation}
	 By assumption, $\norm{W}_{\max} \lesssim N^{-1}$, hence  $\norm{W\vect{e}_a}_\infty \lesssim N^{-1}$.  Using Claim \ref{Pi 1 > c} we get  
	\begin{equation} \label{s^a bound}
			|s_a| \lesssim \frac{N^{-1}\opnorminf{\Pi}}{\norm{\Pi\vect{1}}_{\infty}} \lesssim N^{-1}, \quad a\in \{1,\dots, N\}.
	\end{equation} 
	We combine \eqref{decomp g} and the resolvent identity in the form $(z-\zeta)G\other{G} = G-\other{G}$ to obtain
	\begin{equation} \label{T expr2}
		 \sum_{a\neq y} W_{ax} G_{\alpha a} \other{G}_{a\beta} = \sum_{a\neq y} Y_{ax} G_{\alpha a} \other{G}_{a\beta}  + g_{\alpha\beta}^y \bar{s}_x,\quad 
		\quad g_{\alpha\beta}^y := \frac{G_{\alpha\beta} - \other G_{\alpha\beta}}{z - \zeta} - G_{\alpha y}\other G_{y\beta}.
	\end{equation}
	Define the $N\times N$ matrix $\X:= \left(1 -S\m \other\m \right)^{-1}\Y$.
	It follows from \eqref{decomp g} that $\Y = (1-\Pi)\Y$, hence $X = (1-S\m\other\m)^{-1}(1-\Pi)\Y$. Furthermore,   estimates $\norm{W}_{\max} \lesssim N^{-1}$,  \eqref{decomp g}, and \eqref{s^a bound} imply that $|Y_{ab}|\lesssim N^{-1}$ for all $a$ and $b$. Since by Claim \ref{(1-Smm)(1-Pi)} $\opnorminf{(1-S\m\other\m)^{-1}(1-\Pi)} \lesssim 1$, we conclude that 
	\begin{equation} \label{X_aj_bound}
		\norm{X}_{\max}=\max\limits_{a,b}|\X_{ab}| \lesssim N^{-1}.
	\end{equation}
	First, using \eqref{X_aj_bound}, we can apply Lemma \ref{self_eq_lemma} to the first term in \eqref{T expr2} to obtain
	\begin{equation} \label{self consistent part}
		 \sum_{a\neq y} Y_{ax} G_{\alpha a} \other{G}_{a\beta} = 
		\delta_{\alpha \beta} m_\alpha \other{m}_\alpha ([(1-S\m\other{\m})^{-1}\Y]_{\alpha x} -\delta_{\alpha y}\Y_{\alpha x})  + \Oprec \bigl(\Psi^{2}\other \Psi + \Psi\other\Psi^{2}\bigr).
	\end{equation}
	Using \eqref{decomp g}, we proceed by computing 
	\begin{equation} \label{xy_element_compute}
		\begin{split}
			m_\alpha \other{m}_\alpha[(1-S\m\other{\m})^{-1}\Y]_{\alpha x} =& \bigl[\m\other\m \bigl(1- S\m\other\m \bigr)^{-1}\left( W - \vect{1}\s^*\right)\bigr]_{\alpha x}
			\\=& \bigl[\m\other\m \bigl(1- S\m\other\m \bigr)^{-1}W \bigr]_{\alpha x} 
			-\bigl(\m\other{\m}\bigl(1- S\m\other\m \bigr)^{-1}\vect{1}\bigr)_\alpha\bar{s}_x.
		\end{split}
	\end{equation}
	 Finally, it follows from subtracting the vector Dyson equations \eqref{VD_eq} for $z$ and $\zeta$ that  
	\begin{equation} \label{m_tilde_m_S_identity}
		\m\other{\m}\bigl(1- S\m\other\m \bigr)^{-1}\vect{1} = \frac{\m-\other{\m}}{z-\zeta}.
	\end{equation}
	Next, we estimate the second term in \eqref{T expr2}. Applying the local law in the form \eqref{law1}, we obtain
	\begin{equation} \label{local law part}
		 g_{\alpha\beta}^y  = \delta_{\alpha\beta}\frac{m_\alpha -\other {m}_\alpha}{z-\zeta} - \delta_{\alpha\beta}\delta_{\alpha y}m_\alpha \other {m}_\alpha + \Oprec\bigl((|\eta|+|\other{\eta}|)^{-1}(\Psi+\other\Psi)\bigr), 
	\end{equation}
	where we used that $|z-\zeta| \ge |\eta|+|\other{\eta}|$, since $\eta\other{\eta} <0$.
	Combining \eqref{s^a bound}, \eqref{T expr2}, and \eqref{self consistent part}-\eqref{local law part} yields 
	\begin{equation} 
		\begin{split}
			\sum_{a \neq y} W_{ax} G_{\alpha a} \other{G}_{a\beta} =&\, \delta_{\alpha \beta}\bigl[\m\other\m \bigl(1- S\m\other\m \bigr)^{-1}W \bigr]_{\alpha x} - \delta_{\alpha \beta}\delta_{\alpha y}[\m\other{\m}W]_{\alpha x}\\
			&+ \Oprec \bigl((\Psi+\other\Psi) (\Psi\other \Psi + \min\{\Theta,\other{\Theta}\})\bigr), 
		\end{split}
	\end{equation}
	 which proves \eqref{twoG_law_1} by setting $W_{ax} := w_a$.
	
	To prove \eqref{twoG_law_2}, we observe that  by setting $x=y=\alpha=\beta = b$ in \eqref{T expr2} and summing over $b$ yields
	\begin{equation}
		\sum_{b}\sum_{a\neq b} W_{ab} G_{b a} \other{G}_{ab} = \sum_{b}\sum_{a\neq b} Y_{ab} G_{aa} \other{G}_{ab} + \langle\s, \g\rangle,\quad g_b := \frac{G_{bb} - \other G_{bb}}{z - \zeta} - G_{bb}\other G_{bb}, \,\, b \in \{1,\dots,N\}.
	\end{equation}
	
	To estimate $\langle\s,\g\rangle$, we use \eqref{s^a bound} and the averaged local law \eqref{law2} to obtain
	\begin{equation} \label{trace law part}
		\bigl\langle \s,\g\bigr\rangle = \biggl\langle \s,\frac{\m -\other \m}{z-\zeta} - \m \other \m\biggr\rangle + \Oprec\bigl((|\eta|+|\other{\eta}|)^{-1}(\Theta+\other{\Theta})\bigr),
	\end{equation}
	where we used that $|z-\zeta| \ge |\eta|+|\other{\eta}|$, since $\eta\other{\eta} <0$.
	
	Setting $x=y=\alpha=\beta = b$ in \eqref{self consistent part}, summing over $b$,  using the identities \eqref{xy_element_compute} and \eqref{m_tilde_m_S_identity}, and combining the result with \eqref{trace law part}, we deduce that 
	\begin{equation}
		\begin{split}
			\sum_{b}\sum_{a\neq b} W_{ab} G_{b a} \other{G}_{ab} =& \Tr\bigl[\m\other\m S\m\other\m\bigl(1- S\m\other\m \bigr)^{-1}W\bigr] + N\Oprec\bigl(\Psi\other{\Psi}(\Psi+\other{\Psi})+\Theta\other{\Theta}\bigr),
		\end{split}
	\end{equation} 
	where we used that $(|\eta|+|\other{\eta}|)^{-1}(\Theta+\other{\Theta}) = N\Theta\other{\Theta}$. This establishes \eqref{twoG_law_2} and concludes the proof of Theorem \ref{th_T_local_law}.
\end{proof}
	\begin{remark} \label{optimal_T_rem} 
		We outline the steps needed to achieve the optimal error estimate \eqref{optimal_avgT_error}. 
		First, \new one needs to adapt the proof of Theorem \ref{th_T_local_law}. More specifically, replace the decomposition \eqref{decomp g} \old with
		\begin{equation} \label{optimal_decomp}
			W = Y + \vect{1}\s^* + \vect{q}\vect{1}^*, \text{ such that  }\, \Pi(z,\zeta)Y = Y\Pi^t(\zeta,z) = 0,
		\end{equation} 
		where $\Pi(z,\zeta)$ is the destabilizing eigenprojector defined in \eqref{Pi}. The terms involving $\s$ and $\vect{q}$ are handled using  the averaged  local law \eqref{law2}, similarly to \eqref{trace law part}. 
		
		For the remaining term, $\mathcal{R} := \sum_y \Ff{yy}{yy}$, we \new adapt the mechanism of Lemma \ref{self_eq_lemma} by using the following iterative scheme. In the first step, we apply an expansion similar to \eqref{general_cumul_expan} to the partial derivative $\partial_{jk}\mathcal{R}$. This improves the error in the estimate on $\mathcal{R}$ by a factor of $(\Psi + \other{\Psi})^{1/2}$. If we expand $\partial_{lp}\partial_{jk}\mathcal{R}$ in a similar manner, we gain another $(\Psi + \other{\Psi})^{1/4}$.  Iterating this approach we can estimate $\mathcal{R}$ with an error stochastically dominated by $N\Psi\other{\Psi}(\Psi + \other{\Psi})^{2-2^{-d}}$ for any given integer $d$ (where $d$ is the maximal order of expanded partial derivatives). By Definition \ref{stochdom_def}, this \old is sufficient to establish \eqref{optimal_avgT_error}. Similar arguments in the context of random band matrices can be found in \cite{Erdos2013Fluctuations}.		
	\end{remark}  
\begin{proof}[Proof of Corollary \ref{cor_T_local_law}]
	 Estimate \eqref{T_xy_law} on $T_{xy}(\zeta,z)$ follows from \eqref{twoG_law_1} by setting $\alpha = \beta = y$ and $w_a := S_{xa}$. Estimate \eqref{T_local_law} on $\Tr[AT(z,\zeta)]$ follows from \eqref{twoG_law_2} by setting $W := SA^t$, which satisfies $|W_{ab}| \lesssim  N^{-1}\opnorminf{A}$. This concludes the proof of Corollary \ref{cor_T_local_law}. 
\end{proof}	

\begin{remark}
	Note that estimates \eqref{twoG_law_1} and \eqref{twoG_law_2} (also with the improved error term \eqref{optimal_avgT_error}) hold without omission of indices in the $a$ summation. Indeed, it follows from Theorems \ref{th_T_local_law} and \ref{local_law} that 	 
	\begin{equation}
		\begin{split}
			\sum_{a} w_{a} G_{\alpha a} \other{G}_{a\beta}
			=&\,\delta_{\alpha \beta}\bigl[\m\other{\m}\bigl(1- S\m\other{\m} \bigr)^{-1}\vect{w} \bigr]_{\alpha} + \Oprec \bigl((\Psi+\other{\Psi}) (\Psi\other{\Psi} + \mathds{1}_{\{\eta\other{\eta}<0\}}\min\{\Theta,\other{\Theta}\})\bigr),\\
			\sum_{a,b} W_{ab}G_{ba}\other{G}_{ab} =& \Tr\bigl[\m\other{\m}\bigl(1- S\m\other{\m} \bigr)^{-1}W \bigr]+ \Oprec \bigl(N(\Psi+\other{\Psi})\Psi\other{\Psi}+\mathds{1}_{\{\eta\other{\eta} <0\}}N\Theta\other{\Theta}\bigr).
		\end{split}
	\end{equation}
\end{remark} 
\section{Proof of Proposition \ref{main2}} \label{proof_of_main2}
In this section, we compute the variance $V(f)$ defined in \eqref{variance_V} for mesoscopic $C^2_c$ test functions $f$. 
In \cite{Landon2021Wignertype}, the limiting variance was computed for several types of $C^\infty$ test functions, including compactly supported ones; however, $V(f)$ is computed with an $\mathcal{O}(1)$ error (see, e.g., Lemma 6.7 in \cite{Landon2021Wignertype}), which is not negligible in the setting of the present paper. To obtain effective error bounds, we augment the  proof laid out in \cite{Landon2021Wignertype} by performing further integration by parts in the integral representation of $V(f)$, thus eliminating the $f'$ terms, improving the error by a factor of $\mathcal{O}(\eta_0)$. 

Throughout this section, we adhere to the notation $\m \equiv \m(z), \other\m \equiv \m(\zeta), \eta := \im z, \other{\eta} := \im\zeta$. 

The stability operator $(1- S\m\other{\m})$ can be expressed in terms of the self-saturated energy operator $F$, defined in \eqref{F_def}, via the following identity
\begin{equation} \label{Smm F identity}
	1-S\m\other\m = |\m\other\m|^{-1/2}\left(\mathrm{U}^* - F(z,\zeta)\right)|\m\other\m|^{1/2}\mathrm{U}, \quad \mathrm{U}:= \frac{\m\other\m}{|\m\other\m|}.
\end{equation}
 Furthermore, by \eqref{gapF}, the operator $F$ can be decomposed such that
\begin{equation} \label{F=vv+A}
	F(z,\zeta) = \eigF_1(z,\zeta)\, \vv(z,\zeta)\bigl(\vv(z,\zeta)\bigr)^* + A(z,\zeta), \quad A(z,\zeta)\vv(z,\zeta) = 0, \quad \opnormtwo{A(z,\zeta)} \le 1 - \other{\delta},
\end{equation}
where $\eigF_1,\vv$ is the principal eigenvalue-eigenvector pair of $F$, and $\other{\delta}$ is the constant in \eqref{gapF}. 

Let $R \equiv R(z,\zeta)$ denote $(\mathrm{U}^*(z,\zeta)-A(z,\zeta))^{-1}$. In the sequel, we drop the arguments and write $A\equiv A(z,\zeta)$. Lower bound \eqref{v_bound} and the inequality in \eqref{F=vv+A} imply that 
\begin{equation} \label{norm_R}
	\opnormtwo{R} + \opnorminf{R} \lesssim 1.
\end{equation}
In the following lemma, we collect the perturbative estimates on the saturated self-energy operator $F$ and related quantities established in \cite{Landon2021Wignertype}.
\begin{lemma} (Proposition 6.5, (6.52), (6.60), (6.71), and (6.67) in \cite{Landon2021Wignertype}) Let $w,\zeta_1,\zeta_2$ be spectral parameters in $\bulk_\kappa + \I[-1,1]$, and let $F$ be the operator defined in \eqref{F_def}, then the principal eigenvalue-eigenvector pair $\eigF_1,\vv$ of $F$ satisfies 
	\begin{equation} \label{v,psi_pert} 
		\opnormtwo{\vv(w,\zeta_1) - \vv(w,\zeta_2)} + |\eigF_1(w,\zeta_1) - \eigF_1(w,\zeta_2)| \lesssim |\zeta_1-\zeta_2|.
	\end{equation}	
	Furthermore, for operator $A$ defined in \eqref{F=vv+A}, we have the estimate
	\begin{equation} \label{F_pert} 
		\opnormtwo{F(w,\zeta_1) - F(w,\zeta_2)} + \opnormtwo{A(w,\zeta_1) - A(w,\zeta_2)} \lesssim |\zeta_1-\zeta_2|.
	\end{equation}
	Let $z := x+\I\eta, \zeta := y - \I\eta$, with $x,y \in \bulk_\kappa$, $0 \le \eta \le 1$, then
	\begin{equation} \label{<vRm'mURv>} 
		\eigF_1\bigl\langle \vv,R\frac{\m'}{\m}\mathrm{U}^* R\vv \bigr\rangle = \eigF_1(z,z)\bigl\langle \vv(z,z)\frac{\m'}{\m}\vv(z,z)\bigr\rangle + \mathcal{O}(|x-y|)
	\end{equation}
	Let $\omega \equiv \omega(z,\zeta) := 1-\eigF_1\langle \vv,R\vv\rangle$, then
	\begin{equation} \label{omega_asympt} 
		\omega(z,\zeta) = 1-\eigF_1(z,z) + \eigF_1(z,z)(x-y)\overline{\bigl\langle \vv(z,z)\frac{\m'}{\m}\vv(z,z)\bigr\rangle} + \mathcal{O}(|x-y|^2),
	\end{equation}
	Moreover, there exists $\varepsilon>0$ independent of $N$, such that for all $x,y \in \bulk_\kappa$ satisfying $|x-y| \le \varepsilon$, 
	\begin{equation} \label{omega_lower_bound}
		|\omega(z,\zeta)| \gtrsim \eta + |x-y|.
	\end{equation} 
	Finally, for $z:= x + \I\eta$ with $x\in\bulk_\kappa$, the following identity holds
	\begin{equation} \label{<>formula}
		\lim\limits_{\eta \to +0} \bigl\langle \vv(z,z)\frac{\m'}{\m}\vv(z,z)\bigr\rangle = \frac{\I\pi}{2}\rho(x)\norm{\frac{\im\m(x+\I0)}{|\m(x)|}}_2^{-2}
	\end{equation} 
\end{lemma} 
By our choice of $\kappa$, $E_0$ is in the interior of the bulk interval $\bulk_\kappa$,  defined in \eqref{D_bulk} , hence if we define $\eps:= \min\{\varepsilon/4, \dist(E_0, \mathbb{R}\backslash\bulk_\kappa)\}$, then $\eps \sim 1$. Furthermore, since the function $g$ is compactly supported, we assume that $\supp{f} \subset [E_0-\eps,E_0+\eps]$ for large $N$.	
\begin{lemma} \label{V_good_formula} Let $\eta_* \equiv \eta_*(N)$ satisfy $0<\eta_* \le N^{-100}$, then  $V(f)$, defined in \eqref{variance_V}, admits the estimate 
	\begin{equation} \label{V_tilde_K}
		V(f) 	=\frac{1}{4\pi^2}\iint\limits_{[E_0-\eps,E_0+\eps]^2}(f(y)-f(x))^2\other{\mathcal{K}}(x+\I\eta_*,y-\I\eta_*) \mathrm{d}x\mathrm{d}y +\bigO{\eta_0+N^{-\twoalp}},
	\end{equation}
	where
	\begin{equation} \label{tilde_K}
		\other{\mathcal{K}}(z,\zeta):= -2\re\Tr\left[\frac{\m'}{\m}(1-S\m\other\m)^{-1}S\m\other\m'(1-S\m\other\m)^{-1}\right].
	\end{equation}
\end{lemma}
In preparation for the proof of Lemma \ref{V_good_formula} we define an auxiliary function $\mathcal{L}(z,\zeta)$  
\begin{equation} \label{kernel L}
	\begin{split}
		\mathcal{L}(z,\zeta) &:= \mathcal{L}_{\log}(z,\zeta) +\mathcal{L}_1(z,\zeta),\\
		\mathcal{L}_{\log}(z,\zeta) &:= -2\log\det\left\{1-S\m\other{\m}\right\},\quad
		\mathcal{L}_1(z,\zeta) := -\Tr\left[S\m\other{\m}\right] + \frac{1}{2}\left\langle\overline{\m\other{\m}}, \Cmlnt^{(4)}\m\other{\m} \right\rangle,
	\end{split}
\end{equation} 
where $\log$ is the principal branch of the complex logarithm, and $\Cmlnt^{(4)}$ is the matrix of the fourth cumulants of $H$. By Jacobi's formula for the derivative of the determinant, it follows from the definitions of $\mathcal{L}$ and $\mathcal{K}$,  that for all $z,\zeta \in \mathbb{C}\backslash\mathbb{R}$
\begin{equation} \label{L''=K}
	\frac{\partial^2}{\partial\zeta\partial z} \mathcal{L}(z,\zeta) = \mathcal{K}(z,\zeta).
\end{equation}
Furthermore, by condition \eqref{cond_A} and the upper bound \eqref{m_bound}, it follows that
\begin{equation} \label{L_bound}
	\begin{split}
		|\mathcal{L}_{\log}(z,\zeta)| \le& \pi+\log\left|\det\left\{1-S\m\other\m\right\}\right|
		\lesssim 1+\Tr\left[\left(1-S\m\other\m\right)^*\left(1-S\m\other\m\right) - I\right] \lesssim 1,
	\end{split}
\end{equation}
where in the last line we used $\left[\left(1-S\m\other\m\right)^*\left(1-S\m\other\m\right) - I\right]_{jj}\lesssim N^{-1}$.

 The partial derivatives of $\mathcal{L}_1$ contribute only sub-leading terms to $\mathcal{L}$. Indeed, we have the estimates
\begin{equation} \label{K_small_part}
	\mathcal{L}_1(z,\zeta) \lesssim 1, \quad \frac{\partial}{\partial z}\mathcal{L}_1(z,\zeta) \lesssim 1, \quad \frac{\partial^2}{\partial\zeta\partial z}\mathcal{L}_1(z,\zeta) \lesssim 1,
\end{equation}
where we used the moment condition \eqref{moment_condition} to bound $S_{jk}$ and $\Cmlnt^{(4)}_{jk}$, \eqref{m_bound} to get the upper bound $\m,\other{\m}\lesssim 1$, and \eqref{dm_bound} to obtain $\m',\other{\m}'\lesssim 1$, since $[E_0+\eps,E_0-\eps] \subset \bulk_\kappa$.

The following claim collects the bounds on $\mathcal{K}$ and $\partial_z\mathcal{L}$ that together with \eqref{L_bound} enable integration by parts in the definition \eqref{variance_V} of the variance $V(f)$, which is the essence of Lemma \ref{V_good_formula}.
\begin{claim}\label{K_L_claim} (Proposition 6.2 and Proposition 6.6 in \cite{Landon2021Wignertype}) Let $\mathcal{K}(z,\zeta)$ and $\mathcal{L}(z,\zeta)$ be as defined in \eqref{kernel_K} (with $\beta = 1$) and \eqref{kernel L} respectively, then for all $z,\zeta \in \mathbb{C}\backslash\mathbb{R}$ with $\re z,\re\zeta \in [E_0-\eps,E_0+\eps]$ and $|\im z|, |\im\zeta| \le 1$ we have 
	\begin{equation} \label{K_L'_bounds}
		\mathcal{K}(z,\zeta) \lesssim 1 + \mathds{1}_{\{\eta\other{\eta} <0\}}(|\eta| + |\other{\eta}|)^{-2}, \quad \frac{\partial}{\partial z}\mathcal{L}(z,\zeta) \lesssim 1 + (|\re z - \re\zeta|+|\eta| + |\other{\eta}|)^{-1},
	\end{equation}
	where	 $\eta:=\im z, \other{\eta}:= \im\zeta$.
\end{claim} 
\begin{proof}[Proof of Lemma \ref{V_good_formula}]
	Define $\Omega_* := \{z\in \mathbb{C}: 1>|\im z| > \eta_*\}$. Recall the definition of $V(f)$ from \eqref{variance_V}. First, we prove that 
	\begin{equation} \label{V(f)_ultra_local}
		V(f) = \frac{1}{\pi^2}\int\limits_{\Omega_*}\int\limits_{\Omega_*} \frac{\partial \other f(\zeta)}{\partial \bar \zeta}\frac{\partial \other f(z)}{\partial \bar z} \mathcal{K}(z,\zeta)\mathrm{d}\bar\zeta\mathrm{d}\zeta \mathrm{d}\bar z\mathrm{d}z + \bigO{N^{-\twoalp}}.
	\end{equation}
	It follows from \eqref{QA_f} that
	\begin{equation} \label{quasi_analytic_deriv}
		\frac{\partial\other{f}}{\partial\bar{z}} = \frac{1}{2}\biggl(- \eta \chi'(\eta)  f'(x)  + i \bigl( \eta \chi(\eta)f''(x) + \chi'(\eta) f(x)\bigr) \biggr).
	\end{equation}
	Moreover, for all $z$ with $|\im z|< 1/2$, \eqref{quasi_analytic_deriv} and the properties of $\chi$ in \eqref{QA_f} imply
	\begin{equation}\label{quasi_f_partial_small_eta}
		\frac{\partial \other f}{\partial \bar z} = \frac{i\im z}{2}f''(\re z) .
	\end{equation}
	
	Let $V_*(f)$ denote the integral on right hand side of \eqref{V(f)_ultra_local}, and define $\eta_1 := N^{-\alp}\eta_0$. It follows from the first inequality in \eqref{K_L'_bounds}, and \eqref{quasi_f_partial_small_eta} that  
	\begin{equation} \label{V-V_*}
		|V(f)-V_*(f)|\lesssim \iint\limits_{\mathbb{R}^2}\left|f''(x)f''(y)\right|\mathrm{d}x\mathrm{d}y\int\limits_{\eta_*}^{\eta_1}\int\limits_{\eta_*}^{2\eta_1} \frac{\eta\other{\eta}}{(\eta +\other{\eta})^2}\mathrm{d}\other{\eta}\mathrm{d}\eta.
	\end{equation}
	Note that $\eta\other{\eta} \le (\eta + \other{\eta})^2/4$, hence the integral over $\mathrm{d}\other{\eta}\mathrm{d}\eta$ is bounded by $\eta_1^2/2$, and since $\norm{f''}_1 \sim \eta_0^{-1}$, \eqref{V(f)_ultra_local} is established.
	
	We write $z:= x+\I\eta, \zeta := y+\I\other{\eta}$ and plug \eqref{L''=K} into the expression \eqref{V(f)_ultra_local} for $V(f)$. Using the fact that $\partial_z u = -i\partial_\eta u$ for any holomorphic function $u(z)$, and integrating by parts in $\eta$, we obtain
	\begin{equation} \label{by_parts_in_eta}
		\begin{split}
			V(f) =& \frac{\I}{\pi^2}\iint\limits_{\mathbb{R}^2}\mathrm{d}x\mathrm{d}y\int\limits_{|\other{\eta}|>\eta_*} \frac{\partial \other f(\zeta)}{\partial \bar \zeta}\int\limits_{|\eta|>\eta_*} \frac{\partial^2 \other f(z)}{\partial \eta\partial \bar z} \frac{\partial}{\partial\zeta} \mathcal{L}(z,\zeta) \mathrm{d}\other{\eta}\mathrm{d}\eta\\
			&-\frac{\I}{\pi^2}\iint\limits_{\mathbb{R}^2}\mathrm{d}x\mathrm{d}y\int\limits_{|\other{\eta}|>\eta_*} \frac{\partial \other f(\zeta)}{\partial \bar \zeta}\sum\limits_{\eta=\pm\eta_*} \frac{\partial \other f}{\partial \bar z}(x+\I\eta) \frac{\partial}{\partial\zeta} \mathcal{L}(z,\zeta)\mathrm{d}\other{\eta}+ \bigO{N^{-\twoalp}}.
		\end{split}
	\end{equation}
	The second estimate in \eqref{K_L'_bounds}, expression \eqref{quasi_analytic_deriv} and the estimates $\norm{f''}_1 \sim \eta_0^{-1}, \norm{f'}_1\sim 1, \norm{f}_1\sim\eta_0$ imply that the boundary term in \eqref{by_parts_in_eta} is dominated by $\Oprec(\eta_*\eta_0^{-2})$, which is smaller than $\bigO{N^{-\twoalp}}$.
	
	Similarly, integrating the first term on the right hand side of \eqref{by_parts_in_eta} by parts in $\other{\eta}$ we get
	\begin{equation} \label{by_parts_in_eta'}
		\begin{split}
			V(f) =& -\frac{1}{\pi^2}\int\limits_{\Omega_*}\int\limits_{\Omega_*}\frac{\partial^2 \other f(z)}{\partial \bar z\partial \eta} \frac{\partial^2 \other f(\zeta)}{\partial \bar \zeta\partial\other{\eta}} \mathcal{L}(z,\zeta) \mathrm{d}\bar\zeta\mathrm{d}\zeta \mathrm{d}\bar z\mathrm{d}z\\
			&+\frac{1}{\pi^2}\iint\limits_{\mathbb{R}^2}\mathrm{d}x\mathrm{d}y\int\limits_{|\eta|>\eta_*}  \frac{\partial^2 \other f(z)}{\partial \eta\partial \bar z}\sum\limits_{\other\eta=\pm\eta_*} \frac{\partial \other f}{\partial \bar \zeta}(y+\I\other{\eta}) \mathcal{L}(z,y+\I\other{\eta})\mathrm{d}\eta+ \bigO{N^{-\twoalp}}.
		\end{split}
	\end{equation}
	It follows from \eqref{L_bound} and the expression \eqref{quasi_analytic_deriv} that the boundary term (the second line of \eqref{by_parts_in_eta'}) is again dominated by $\Oprec(N^{-\twoalp})$.
	
	We apply Stokes' theorem to \eqref{by_parts_in_eta'} twice: once in $z$ and once in $\zeta$. Considering that $\partial_\eta \other{f}(z)$ vanishes on the boundary of $\Omega_*$ except for the lines $\{\im z = \pm\eta_*\}$, this results in
	\begin{equation} \label{V(f)_f'f'_form}
		\begin{split}
			V(f) =& \frac{1}{4\pi^2}\iint\limits_{\mathbb{R}^2}\sum\limits_{\eta,\other{\eta} = \pm\eta_*}\sign\left(\eta\other{\eta}\right)\frac{\partial \other f(x+\I\eta)}{\partial \eta} \frac{\partial \other f(y+\I\other{\eta})}{\partial\other{\eta}} \mathcal{L}(x+\I\eta,y+\I\other{\eta}) \mathrm{d}x\mathrm{d}y + \bigO{N^{-\twoalp}}\\
			=&-\frac{1}{2\pi^2}\iint\limits_{\mathbb{R}^2}f'(x)f'(y)\other{\mathcal{L}}(x,y) \mathrm{d}x\mathrm{d}y + \bigO{N^{-\twoalp}},
		\end{split}
	\end{equation}  
	where 
	\begin{equation}
		\other{\mathcal{L}}(x,y):=\re\left[\mathcal{L}(x+\I\eta_*,y+\I\eta_*)-\mathcal{L}(x+\I\eta_*,y-\I\eta_*)\right]
	\end{equation}
	 We restrict the integrations in \eqref{V(f)_f'f'_form} to $[E_0-\eps, E_0+\eps]$, since this interval contains the support of $f$. Furthermore, for all $y\in\supp{f}$, $y-E_0\lesssim \eta_0$, hence $|y-E_0\pm\eps|\sim 1$. By symmetry of $\mathcal{L}(z,\zeta)$, and the second estimate in \eqref{K_L'_bounds} it follows that  
	\begin{equation} \label{boundary_L'_bound}
		\frac{\partial}{\partial y}\other{\mathcal{L}}(E_0\pm\eps,y) \lesssim 1, \quad y \in \supp{f}.
	\end{equation} 
	We write $f'(y) = \partial_y\left(f(y)-f(x)\right)$, perform integration by parts in $y$ and integrate the boundary term by parts in $x$ to obtain
	\begin{equation} \label{by_parts_in_y}
		\begin{split}
			V(f)
			=&\frac{1}{2\pi^2}\int\limits_{E_0-\eps}^{E_0+\eps}\int\limits_{E_0-\eps}^{E_0+\eps}f'(x)\left(f(y)-f(x)\right)\frac{\partial}{\partial y} 
			\other{\mathcal{L}}(x,y)
			\mathrm{d}x\mathrm{d}y\\
			&+\frac{1}{4\pi^2}\int\limits_{E_0-\eps}^{E_0+\eps} (f(x))^2 \frac{\partial}{\partial x}\left(\other{\mathcal{L}}(x,E_0+\eps) -\other{\mathcal{L}}(x,E_0-\eps) \right) \mathrm{d}x + \bigO{N^{-\twoalp}}.
		\end{split}
	\end{equation}
	Since $\norm{f}_2^2 \lesssim \eta_0$,  it follows from \eqref{boundary_L'_bound}  that the second integral in \eqref{by_parts_in_y} is $\bigO{\eta_0}$. 
	Similarly, integrating \eqref{by_parts_in_y} by parts in $x$ and using \eqref{by_parts_in_y} to substitute one of the emerging itegrals for $-V(f) + \bigO{N^{-\twoalp}+\eta_0}$, we get
	\begin{equation} \label{by_parts_in_x}
		\begin{split}
			2V(f)
			=&\frac{1}{2\pi^2}\int\limits_{E_0-\eps}^{E_0+\eps}\int\limits_{E_0-\eps}^{E_0+\eps}\left(f(y)-f(x)\right)^2\frac{\partial^2}{\partial x\partial y} 
			\other{\mathcal{L}}(x,y)
			\mathrm{d}x\mathrm{d}y
			+ \bigO{\eta_0+N^{-\twoalp }},
		\end{split}
	\end{equation}
	where we again used \eqref{boundary_L'_bound} to estimate the boundary term. 
	For any holomorphic function $u(z)$ of $z = x + i\eta$, we have $\partial_x u = \re[\partial_z u]$, hence $\partial_x\partial_y\other{\mathcal{L}}(x,y) = \re\left[\mathcal{K}(x+\I\eta_*,y+\I\eta_*)-\mathcal{K}(x+\I\eta_*,y-\I\eta_*)\right]$.
	
	Finally, in view of in view of the first estimate in \eqref{K_L'_bounds}, $\partial_z\partial_\zeta \mathcal{L}_{log}(x+\I\eta_*,y+\I\eta_*) \lesssim 1$, so its contribution is also bounded by $\Oprec(\eta_0\norm{g}^2_2 + \eta_0^2\norm{g}_1^2 )$.  Moreover, it follows from the last estimate in \eqref{K_small_part} that we can replace $\mathcal{K}(x+\I\eta_*,y-\I\eta_*)$ by $\partial_z\partial_\zeta \mathcal{L}_{log}(x+\I\eta_*,y-\I\eta_*)$, since the contribution of the remaining terms is bounded by $\Oprec(\eta_0\norm{g}^2_2 + \eta_0^2\norm{g}_1^2 )$. This concludes the proof of Lemma \ref{V_good_formula}.
\end{proof}

Once Lemma \ref{V_good_formula} is established, we can follow the method of Lemma 6.7 in \cite{Landon2021Wignertype} to finish the proof of Proposition \ref{main2}.

Fix $x,y \in [E_0-\eps, E_0+\eps]$ and write $z:= x + \I \eta_*$, $\zeta:= y - \I\eta_*$, as in \eqref{V_tilde_K}.
It follows from \eqref{Smm F identity} and \eqref{F=vv+A} that the kernel $\other{\mathcal{K}}(z,\zeta)$ can be written as
\begin{equation} \label{K_rep}
	\other{\mathcal{K}}(z,\zeta) = -2\re\Tr\biggl[\frac{\m'}{\m}\mathrm{U}^*\bigl(R + \frac{\eigF_1}{\omega} R\vv\vv^*R\bigr)F\frac{\other{\m}'}{\other{\m}}\bigl(R + \frac{\eigF_1}{\omega} R\vv\vv^*R\bigr) \biggr],
\end{equation}
where $\omega$ is defined in \eqref{omega_asympt}.
Expanding the brackets in \eqref{K_rep}, collecting like terms according to the powers of $\omega^{-1}$, and using the cyclic property of trace yields
\begin{equation} \label{K_leading}
	\other{\mathcal{K}}(z,\zeta) = -2\re\biggl[\frac{\eigF_1^2}{\omega^2} \bigl\langle \vv,R\frac{\m'}{\m}\mathrm{U}^* R\vv \bigr\rangle\bigl\langle\vv,RF\frac{\other{\m}'}{\other{\m}} R\vv\bigr\rangle\biggr] + \bigO{1+\omega^{-1}},
\end{equation}
since $\Tr\bigl[\frac{\m'}{\m}\mathrm{U}^*RF\frac{\other{\m}'}{\other{\m}}R\bigr]$, $\Tr\bigl[\frac{\m'}{\m}\mathrm{U}^*RF\frac{\other{\m}'}{\other{\m}}R\vv\vv^*R\bigr]$, and $\Tr\bigl[\frac{\m'}{\m}\mathrm{U}^*R\vv\vv^*RF\frac{\other{\m}'}{\other{\m}}R\bigr]$ are all $\mathcal{O}(1)$. The first scalar product in \eqref{K_leading} can be estimated using \eqref{<vRm'mURv>}.

We compute the second scalar product in \eqref{K_leading}. It follows from uniform bounds \eqref{m_bound} and \eqref{dm_bound} that $\lVert \m(z)-\m(\bar \zeta)\rVert_\infty \lesssim |x-y|$, and hence $\opnormtwo{\mathrm{U}(z,\zeta) -1} \lesssim |x-y|$. Together with estimates \eqref{F_pert} and \eqref{v,psi_pert}, this yields
\begin{equation} \label{second<>}
	\eigF_1\bigl\langle\vv,RF\frac{\other{\m}'}{\other{\m}}R\vv\bigr\rangle = \langle\vv(\zeta,\zeta),F(\zeta,\zeta)\frac{\other{\m}'}{\other{\m}}\vv(\zeta,\zeta)\bigr\rangle + \mathcal{O}(|x-y|),
\end{equation}
where we used the identity $R(\bar\zeta,\zeta)\vv(\zeta,\zeta) = (1-A(\zeta,\zeta))^{-1}\vv(\zeta,\zeta) = \vv(\zeta,\zeta)$.

It follows from the estimate on $\vv$ in \eqref{v,psi_pert} that $\norm{\vv(\zeta,\zeta) - \vv(y,y)}_2 \lesssim \eta_*$. Vector $\vv(y,y)$ is the $\ell^2$-normalization of $|\m(y)|^{-1}\im\m(y+\I 0)$, hence it satisfies $F(y,y) \vv(y,y) = \vv(y,y)$ by \eqref{VD_eq}. Therefore using \eqref{dm_bound} and the lower bound in \eqref{F_pert}, we obtain 
\begin{equation} \label{v_approx}
	\norm{F(\zeta,\zeta)\vv(\zeta,\zeta) - \vv(\zeta,\zeta)}_2 \lesssim \eta_*.
\end{equation}
Substituting \eqref{v_approx} into \eqref{second<>} yields
\begin{equation} \label{second<>est}
	\eigF_1\bigl\langle\vv,RF\frac{\other{\m}'}{\other{\m}}R\vv\bigr\rangle = \langle \vv(\zeta,\zeta),\frac{\other{\m}'}{\other{\m}}\vv(\zeta,\zeta)\bigr\rangle + \mathcal{O}(|x-y| + \eta_*),
\end{equation}
Combining \eqref{K_rep} with estimates \eqref{v,psi_pert}, \eqref{<vRm'mURv>}, \eqref{omega_lower_bound} and \eqref{second<>est} yield
\begin{equation} \label{K_estimate}
	\other{\mathcal{K}}(z,\zeta) = -2\re \left[\frac{\eigF_1(z,z)\eigF_1(\zeta,\zeta)}{\omega^2}\bigl\langle \vv(z,z)\frac{\m'}{\m}\vv(z,z)\bigr\rangle\langle \vv(\zeta,\zeta),\frac{\other{\m}'}{\other{\m}}\vv(\zeta,\zeta)\bigr\rangle\right] +  \mathcal{O}(1+\omega^{-1}).
\end{equation}
It follows by \eqref{<>formula} and \eqref{omega_asympt} that
\begin{equation} \label{K_limit}
	\lim\limits_{\eta_*\to +0} \other{\mathcal{K}}(x+\I\eta_*,y-\I\eta_*) = 2|x-y|^{-2} + \mathcal{O}(|x-y|^{-1}).
\end{equation}
Since $f\in C^2_c(\mathbb{R})$, \eqref{K_estimate} implies that the integrand in \eqref{V_tilde_K} is uniformly bounded in $\eta_* \in [0, N^{-100}]$. Therefore, we can take the limit $\eta_* \to 0$ in \eqref{V_tilde_K}, and apply the boundary estimate \eqref{K_limit} to obtain.
\begin{equation}
	V(f) = \frac{1}{2\pi^2}\iint\limits_{[E_0-\eps,E_0+\eps]^2}\frac{(f(x)-f(y))^2}{(x-y)^2}\mathrm{d}x\mathrm{d}y +\bigO{\eta_0\log N + N^{-\twoalp}},
\end{equation}
because the contribution of $\mathcal{O}(|x-y|^{-1})$ to the integral \eqref{V_tilde_K} is bounded by $\mathcal{O}(\eta_0\log N)$. 

Finally, the contribution of the regime $(x,y) \notin [E_0-\eps,E_0+\eps]^2$ to the integral
\begin{equation} \label{H1/2_norm_f_norm_g_int}
	\iint\limits_{\mathbb{R}^2}\frac{(f(x)-f(y))^2}{(x-y)^2}\mathrm{d}x\mathrm{d}y = \norm{f}_{\dot{H}^{1/2}}^2 = \norm{g}_{\dot{H}^{1/2}}^2,
\end{equation} is bounded by $\Oprec(\eta_0)$, therefore
\begin{equation}
	V(f) = \frac{1}{2\pi^2}\norm{g}_{\dot{H}^{1/2}}^2 +\bigO{\eta_0\log N+N^{-\twoalp}}.
\end{equation}
This concludes the proof of Proposition \ref{main2}. 

\appendix
\gdef\thesection{Appendix \Alph{section}}
\section{Proof of Lemma \ref{standard_estimate_lemma}} \label{App1}
\gdef\thesection{\Alph{section}}
We use the Helffer–Sjöstrand representation to express the linear eigenvalue statistics in terms of the resolvent of $H$ (see Section 4.2 in \cite{Landon2020applCLT} for references),
   \begin{equation} \label{1-expvTrf}
	\{1 -\Expv\}\left[\Tr f(H)\right] = \frac{1}{2\pi}\int\limits_{\mathbb{C}} \frac{\partial\other{f}}{\partial\bar{z}} \{1-\Expv\}\left[\Tr G(z)\right]\mathrm{d}\bar z \mathrm{d}z.
\end{equation}
The characteristic function $\phi$ then admits the form
\begin{equation} \label{phi_and_e}
	\phi(\lambda) = \E{e(\lambda)}, \quad e(\lambda) := \exp\biggl\{\I\lambda\frac{1}{2\pi} \int\limits_{\mathbb{C}}\frac{\partial \other{f}}{\partial \bar{z}}\{1-\Expv\}\left[\Tr G(z)\right] \mathrm{d}\bar z \mathrm{d}z \biggr\},\quad \lambda \in \mathbb{R},
\end{equation}
and its derivative $\phi'$ is given by
\begin{equation}  \label{phi' expr}
	\phi'(\lambda) = \Expv\biggl[e(\lambda) \frac{\I}{2\pi}\int\limits_{\mathbb{C}} \frac{\partial\other{f}}{\partial\bar{z}}
	\left\{1-\Expv\right\}\left[ \Tr G(z)\right] \mathrm{d}\bar z \mathrm{d}z\biggr],\quad \lambda \in \mathbb{R}.
\end{equation}
As observed in \cite{Landon2020applCLT}, the regime $|\im z| \le N^{-\alp}\eta_0$, referred to as the \textit{ultra-local scales}, does not contribute to the integrals in \eqref{phi_and_e} and \eqref{phi' expr}. This yields the estimates \eqref{phi'_Omega_int} (see equations (4.21) and (4.22) in \cite{Landon2020applCLT} for further detail).

It remains to show that \eqref{trace_fluctuation} holds. Applying the cumulant expansion formula \eqref{cumulant_formula} to the quantity $\E{\other e(\lambda)\left\{1-\Expv\right\}\left[G_{jj}(z)\right]}$ yields the following lemma.
\begin{lemma} (Lemma 5.7 in \cite{Landon2021Wignertype}) \label{lemma 1-E Gjj}
	For all $z \in \mathcal{D}$ defined in \eqref{D_def} and $j\in \{1,\dots,N\}$ we have
	\begin{equation} \label{1/mj(z) proposition}
		\begin{split}
			\frac{-1}{m_j(z)}\E{\other e(\lambda)\left\{1-\Expv\right\}\left[ G_{jj}(z)\right]}
			=& -m_j(z) \sum\limits_{k=1}^N S_{jk}\E{\other e(\lambda)\left\{1-\Expv\right\}\left[G_{kk}(z)\right]}\\
			&-\E{\other e(\lambda)\left\{1-\Expv\right\}\left[ T_{jj}(z,z)\right]}\\
			&+\Expv\biggl[\sum\limits_{k=1}^N S_{jk}G_{kj}(z) \frac{\partial\other e(\lambda)}{\partial H_{jk}}\biggr]\\
			&-\frac{1}{2}\sum\limits_{k=1}^N \Cmlnt^{(4)}_{jk}m_j(z)m_k(z)\Expv\biggl[\frac{\partial^2\other e(\lambda)}{\partial H_{jk}^2}\biggr]\\
			&+\Oprec\left((1+|\lambda|^4)\left(\Psi(z)\Theta(z) + N^{-1} \Psi(z)\eta_0^{-1/2}\right)\right),
		\end{split}
	\end{equation}
	where $\eta_0$ is from \eqref{scaled_f}, and for $a,b \in \{1,\dots, N\}$, $z,\zeta \in \mathbb{C}\backslash\mathbb{R}$, $T_{xy}(z,\zeta)$ is defined in \eqref{Tfunction}.
\end{lemma}
Let $\flucG_j := \E{\other e(\lambda)\left\{1-\Expv\right\}\left[ G_{jj}(z)\right]}$ and let $\flucRHS_j$ denote the right-hand side of \eqref{1/mj(z) proposition} without the first term, then \eqref{1/mj(z) proposition} reads $\left[\left(1-S\m^2(z)\right)\flucG\right]_j = -m_j(z)\flucRHS_j$. The operator $\left(1-S\m^2(z)\right)$ can be inverted to deduce that $\flucG_j = -\left[\left(1-S\m^2(z)\right)^{-1}\m(z)\flucRHS\right]_j$, where $\m(z)$ is interpreted as a multiplication operator acting on the vector $\flucRHS$. Summing over $j$, we obtain 
\begin{equation} \label{fluct_tr_G}
	\E{\other e(\lambda)\left\{1-\Expv\right\}\left[ \Tr G(z)\right]} = \sum\limits_{j=1}^N\flucG_j = -\sum\limits_{j,k=1}^N\left[\left(1-S\m^2(z)\right)^{-1}\right]_{jk}m_k(z)\flucRHS_{\,k} = -\sum\limits_{j=1}^N \frac{m_j'(z)}{m_j(z)}\flucRHS_j,
\end{equation}
where in the last step we applied the identity $\m'(z)/\m^2(z) = (1-S\m^2(z))^{-1}\vect{1}$. The second term on the right-hand side of \eqref{1/mj(z) proposition} contributes the first term to the right hand side of \eqref{trace_fluctuation}, which, as we show in Section \ref{Curly T section}, is negligible. Therefore, it suffices to estimate the contribution of the third and fourth terms on the right-hand side. The necessary estimates on the partial derivatives of $\other{e}(\lambda)$ are collected in the following lemma.
\begin{lemma} (Lemma 5.6 in \cite{Landon2021Wignertype}) \label{tilde_e_partial_lemma}
	For all $j,k \in \{1,\dots,N\}$ we have
	\begin{equation} \label{tilde_e_partial_formula}
		\frac{\partial\other e(\lambda)}{\partial H_{jk}} = -\frac{\I\lambda}{\pi}\frac{2}{1+\delta_{jk}}\other e(\lambda) \int\limits_{\dom'}\frac{\partial\other{f}}{\partial\bar\zeta} \frac{\partial G_{kj}(\zeta)}{\partial \zeta} \mathrm{d}\bar \zeta \mathrm{d}\zeta.
	\end{equation}
	Moreover, for all $p \in \mathbb{N}$, the following bound holds
	\begin{equation} \label{tilde_e_partial_universal_bound}
		\biggl|\frac{\partial^p\other{e}(\lambda)}{\partial H_{jk}^p} \biggr| = \Oprec\bigl((1+|\lambda|)^p\bigr),
	\end{equation}
	and for $k\neq j$
	\begin{equation} \label{tilde_e_partial_jk_estimate}
		\left|\frac{\partial\other e(\lambda)}{\partial H_{jk}}\right| = \Oprec\bigl(N^{-1/2}(1+|\lambda|)\eta_0^{-1/2}\bigr).
	\end{equation}
	Second derivatives with $k\neq j$ are given by
	\begin{equation} \label{tilde_e_partial2_formula}
		\frac{\partial^2\other e(\lambda)}{\partial H_{jk}^2} = \frac{2\I\lambda}{\pi}\other e(\lambda) \int\limits_{\dom'}\frac{\partial\other{f}}{\partial\bar\zeta} \frac{\partial\left\{ m_j(\zeta)m_k(\zeta)\right\}}{\partial \zeta} \mathrm{d}\bar \zeta \mathrm{d}\zeta + \Oprec\bigl(N^{-1/2}(1+|\lambda|)^2\eta_0^{-1/2}\bigr).
	\end{equation}
\end{lemma}
The form in which we write the error terms in Lemmas \ref{lemma 1-E Gjj} and \ref{tilde_e_partial_lemma} slightly differs from their original form in \cite{Landon2021Wignertype} because we have already applied the estimate $\norm{f''}_1\sim \eta_0^{-1}$. The leading term in \eqref{tilde_e_partial2_formula} results in the third line of \eqref{trace_fluctuation}.

Using Lemmas \ref{tilde_e_partial_lemma} and \ref{int_lemma} we proceed to estimate the third term on the right hand side of \eqref{1/mj(z) proposition}.
\begin{lemma} (c.f. Equation (5.65) of Lemma 5.8 in \cite{Landon2021Wignertype}) \label{tilde_e_partial_1}
	For all $z \in \mathcal{D}$ defined in \eqref{D_def} and all $j\in \{1,\dots,N\}$ we have
	\begin{equation} \label{tilde_e_partial_1_eq}
		\begin{split}
			\Expv\biggl[\sum\limits_{k=1}^N S_{jk}G_{kj}(z) \frac{\partial\other e(\lambda)}{\partial H_{jk}}\biggr] =& -\frac{2\I\lambda}{\pi} \Expv\biggl[\other e(\lambda) \int\limits_{\dom'} \frac{\partial\other{f}}{\partial\bar\zeta} \frac{\partial T_{jj}(z,\zeta)}{\partial\zeta} \mathrm{d}\bar \zeta \mathrm{d}\zeta\biggr] \\
			&-\frac{\I\lambda}{\pi} S_{jj} \E{\other e(\lambda)} \int\limits_{\dom'} \frac{\partial\other{f}}{\partial\bar\zeta} m_j'(\zeta) m_j(z) \mathrm{d}\bar \zeta \mathrm{d}\zeta
			+\Oprec\biggl(\frac{\Psi(z)(1+|\lambda|)}{N\eta_0^{1/2}}\biggr).
		\end{split}
	\end{equation}
\end{lemma}
\begin{proof}[Proof of Lemma \ref{tilde_e_partial_1}]
	In view of \eqref{Tfunction}, multiplying \eqref{tilde_e_partial_formula} by $S_{jk}G_{kj}(z)$, summing over $k \neq j$ and taking expectations gives the first term on the right hand side of \eqref{1/mj(z) proposition}.
	For the remaining $k=j$ term, observe that the function $K(\zeta) := G_{jj}(\zeta) - m_j(\zeta)$ is analytic in $\mathbb{C}\backslash\mathbb{R}$ and is stochastically dominated by $\Psi(\zeta)$ in $\mathcal{D}$.
	Applying Lemma \ref{derivative_lemma} with $p=1$ to $K(\zeta)$, we obtain
	\begin{equation} \label{partial_G_jj_local_law}
		\frac{\partial G_{jj}(\zeta)}{\partial\zeta} = m_j'(\zeta) + \Oprec\bigl(|\im\zeta|^{-1}\Psi(\zeta)\bigr).
	\end{equation}
	Plugging \eqref{partial_G_jj_local_law} into \eqref{tilde_e_partial_formula} with $k=j$ and applying Lemma \ref{int_lemma} with $K(\zeta) := \partial_\zeta G_{jj}(\zeta) - m_j'(\zeta)$ with $s=3/2$, we get
	\begin{equation} \label{tilde_e_partial_jj_estimate}
		\frac{\partial\other e(\lambda)}{\partial H_{jj}} = -\frac{\I\lambda}{\pi}\other e(\lambda) \int\limits_{\dom'}\frac{\partial\other{f}}{\partial\bar\zeta} m_j'(\zeta) \mathrm{d}\bar \zeta \mathrm{d}\zeta + \Oprec\bigl(1+|\lambda|)N^{-1/2}\eta_0^{-1/2}\bigr).
	\end{equation}
	where we used the the fact that $|e(\lambda)| = 1$ and the first line of \eqref{phi'_Omega_int} to bound $|\other{e}(\lambda)|$ by $\Oprec(1)$.
	Multiplying \eqref{tilde_e_partial_jj_estimate} by $S_{jj}G_{jj}(z)$ and using the local law \eqref{law1} to estimate $G_{jj}(z)$ gives the second term on the right hand side of \eqref{1/mj(z) proposition}. Application of the local law \eqref{law1} is justified by \eqref{tilde_e_partial_universal_bound} with $p=1$. This concludes the proof of Lemma \ref{tilde_e_partial_1}.
\end{proof}  
Summing up the leading terms in \eqref{tilde_e_partial_1_eq} results in the second and third terms on the right-hand side of \eqref{trace_fluctuation}.
Collecting all the error terms, the estimate in \eqref{trace_fluctuation} now follows from  \eqref{m2S_bound}, \eqref{fluct_tr_G}, \eqref{tilde_e_partial_universal_bound} \eqref{tilde_e_partial2_formula} and Lemma \ref{tilde_e_partial_1}. This concludes the proof of Lemma \ref{standard_estimate_lemma}.


\printbibliography

\end{document}

%% file: commands.tex
\newcommand{\Epsilon}{\mathcal{E}}
\DeclareMathOperator{\Tr}{Tr}

\newcommand\ddfrac[2]{\frac{\displaystyle #1}{\displaystyle #2}}
\newcommand{\norm}[1]{\left\lVert#1\right\rVert}
\newcommand{\opnormtwo}[1]{\norm{#1}_{\ell^2\to\ell^2}}
\newcommand{\opnorminf}[1]{\norm{#1}_{\ell^\infty\to\ell^\infty}}

\newcommand{\m}{\vect{m}}
\newcommand{\supp}[1]{\operatorname{supp}(#1)}
\DeclareMathOperator{\im}{Im}
\DeclareMathOperator{\re}{Re}
\DeclareMathOperator{\prob}{\mathbb{P}}
\DeclareMathOperator{\Expv}{\mathbb{E}}
\DeclareMathOperator{\dist}{dist}
\DeclareMathOperator{\sign}{sign}

\newcommand{\Gap}[1]{\operatorname{Gap}\left(#1\right)}
\newcommand{\diag}[1]{\operatorname{diag}\left(#1\right)}

\newcommand{\bigO}[1]{\mathcal{O}\left(#1\right)}
\newcommand{\smallo}[1]{o\left(#1\right)}
\newcommand{\Oprec}{\mathcal{O}_\prec}

\newcommand{\E}[1]{\Expv\left[#1 \right]}
\newcommand{\Prob}[1]{\prob\left[#1 \right]}
\newcommand{\other}[1]{\widetilde{#1}}
\newcommand{\G}{\mathcal{G}}
\newcommand{\X}{X}
\newcommand{\Y}{Y}
\newcommand{\s}{\vect{s}}
\newcommand{\g}{\vect{g}}
\newcommand{\al}{\mathfrak{a}}
\newcommand{\be}{\mathfrak{b}}
\newcommand{\vv}{\vect{v}}

\newcommand{\eigF}{\psi} 
\newcommand{\eps}{\hat\varepsilon}
\newcommand{\alp}{{\varepsilon_0/2}} 
\newcommand{\twoalp}{\varepsilon_0} 
\newcommand{\dom}{\Omega_0}
\newcommand{\vect}[1]{\mathbf{#1}}
\newcommand{\bulk}{\mathcal{I}}
\newcommand{\I}{i}
\newcommand{\stab}{B}
\newcommand{\annulus}{\mathcal{A}}

\newcommand{\Cmlnt}{\mathcal{C}}
\DeclareFontFamily{OT1}{pzc}{}
\DeclareFontShape{OT1}{pzc}{m}{it}{ <-> s*[1.1] pzcmi7t }{}
\DeclareMathAlphabet{\mathpzc}{OT1}{pzc}{m}{it}
\newcommand{\flucG}{\mathpzc{g}}
\newcommand{\flucRHS}{\mathpzc{r}}

\newcommand{\Ff}[2]{\mathcal{F}^{#1}_{#2}}

\newtheorem{theorem}{Theorem}
\numberwithin{theorem}{section}
\newtheorem{lemma}[theorem]{Lemma}
\newtheorem{Def}[theorem]{Definition}
\newtheorem{prop}[theorem]{Proposition}
\newtheorem{claim}[theorem]{Claim}
\newtheorem{remark}[theorem]{Remark}
\newtheorem{corollary}[theorem]{Corollary}

\newcommand{\new}{}
\newcommand{\old}{\color{black}}